\documentclass[amscd,amssymb,verbatim,10pt]{amsart}

\title[Stratification of free boundary points]
{Stratification of free boundary points for a two-phase variational problem}

\usepackage{esint}
\usepackage{xcolor}
\usepackage[mathscr]{euscript}

\usepackage{tikz}
\usetikzlibrary{shapes,arrows}
\usetikzlibrary{positioning,shadows,arrows}

\usepackage{caption}
\usepackage{subcaption}

\theoremstyle{plain}
\newtheorem{theorem}{Theorem}[section]
\newtheorem{lemma}[theorem]{Lemma}

\newtheorem{prop}[theorem]{Proposition}
\newtheorem{corollary}[theorem]{Corollary}
\newtheorem{remark}[theorem]{Remark}
\theoremstyle{plain}
\newtheorem{defn}[theorem]{Definition}
\theoremstyle{plain}
\newtheorem*{AAA}{Theorem A}

\newtheorem*{BBB}{Theorem B}
\newtheorem*{CCC}{Theorem C}

\renewcommand\epsilon\varepsilon 
\renewcommand\phi\varphi 
\renewcommand\div{\operatorname{div}} 
\newcommand\dist{\operatorname{dist}} 
\newcommand\R{\mathbb{R}} 
\newcommand\N{\mathbb{N}}
\numberwithin{equation}{section}
\newcommand\e{{\epsilon}}
\newcommand\na{{\nabla}}
\newcommand\fb[1]{\partial\{{#1>0}\}}
\newcommand\fbr[1]{\partial_{\text{red}}\{{#1>0}\}}

\renewcommand\L{\rule[-0.5mm]{0.5pt}{4mm}\rule[-0.5mm]{4mm}{0.5pt}}

\usepackage[english]{babel}
\usepackage{blindtext}
\usepackage[doublespacing]{setspace}
\usepackage{amsmath, amssymb}

\usepackage{comment}

\usepackage{graphicx, geometry}
\usepackage[a4,center]{crop}

\author[S. Dipierro]{Serena Dipierro}
\address[Serena Dipierro]{Maxwell Institute for Mathematical Sciences and
School of Mathematics, University of Edinburgh,
James Clerk Maxwell Building, Peter Guthrie Tait Road,
Edinburgh EH9 3FD,
United Kingdom, and Otto-von-Guericke-Universit\"at Magdeburg, Fakult\"at f\"ur Mathematik, 
Institut f\"ur Analysis und Numerik, Universit\"atsplatz 2, 39106 Magdeburg, Germany
}
\email{serena.dipierro@ed.ac.uk}

\author[A.L. Karakhanyan]{Aram L. Karakhanyan}
\address[Aram L. Karakhanyan]{Maxwell Institute for Mathematical Sciences and
School of Mathematics, University of Edinburgh,
James Clerk Maxwell Building, Peter Guthrie Tait Road,
Edinburgh EH9 3FD,
United Kingdom.
}
\email{aram.karakhanyan@ed.ac.uk}

\thanks{2000 Mathematics Subject Classification. Primary 35R35, 35J92. 
Keywords: Free boundary regularity, two phase, $p-$Laplace, partial regularity. 
\\
The first author was supported by EPSRC grant EP/K024566/1 and Alexander von Humboldt Foundation. 
The second author was partially supported by EPSRC grant EP/K024566/1.
}

\newcommand\X[1]{{\chi_{\{#1\}}}}

\begin{document}
\maketitle

\begin{abstract}
In this paper we study the two-phase Bernoulli type free boundary problem
arising from the minimization of the functional
$$J(u):=\int_{\Omega}|\nabla u|^p +\lambda_+^p\,\chi_{\{u>0\}} +\lambda_-^p\,\chi_{\{u\le 0\}},
\quad 1<p<\infty.
$$
Here $\Omega\subset\R^N$ is a bounded smooth domain and $\lambda_\pm$ are 
positive constants such that $\lambda_+^p-\lambda^p_->0$.
We prove the following dichotomy: if $x_0$ is a free boundary point 
then either the free boundary is smooth near $x_0$ or
$u$ has linear growth at $x_0$. Furthermore, 
we show that for  $p>1$ the free boundary has locally finite perimeter
and the 
set of non-smooth points of free boundary is of zero $(N-1)$-dimensional Hausdorff measure. 
Our approach is new even for the classical case $p=2$.
\end{abstract}

\setcounter{tocdepth}{1}
\tableofcontents
%
%
\section{Introduction}
In this paper we study the local minimizers of 
\begin{equation}\label{Ju}
J(u):=\int_{\Omega}|\nabla u|^p +\lambda_+^p\,\chi_{\{u>0\}} +\lambda_-^p\,\chi_{\{u\le 0\}},\quad 
u\in \mathcal A,
\end{equation}
where $\Omega$ is a bounded and smooth domain in $\R^N$, 
$\chi_D$ is the characteristic function of the set $D\subset \R^N$, and
$\lambda_\pm$ are positive constants such that 
\begin{equation}\label{big-lam}
\Lambda:=\lambda_+^p-\lambda_-^p>0.
\end{equation}
The class of admissible functions  $\mathcal A$ consists 
of those functions $u\in W^{1, p}(\Omega)$, with $1<p<\infty$, 
such that $u-g\in W^{1, p}_0(\Omega)$ for a given boundary datum $g$.

\smallskip 

This type of problems arises in jet flow models with two ideal fluids, 
see e.g.~\cite{BZ} and \cite{Gur} page 126, and has been 
studied in \cite{ACF} for $p=2$. When the velocity ${\bf v}$ of the planar flow  
depends on the gradient of the stream function $u$ in power law 
${\bf v}=|\nabla u|^{p-2}\nabla u$ (see \cite{Ast}), then the 
resulted problem for steady state admits a variational formulation with the functional \eqref{Ju}.
In higher dimensions, this models heat (or electrostatic) energy optimization under power Fourier law, 
see \cite{Philip}.

\smallskip

For admissible functions in $\mathcal A^+=\{u\in \mathcal A, u\ge 0\}$ the analogous problem 
has been studied in \cite{DP2}.  
However, the two-phase problem for general growth functionals has remained fundamentally open. 
Towards this direction there are only some partial results available under the assumption of small Lebesgue density 
on the negative phase, see~\cite{K1, Moreira}. 
This is due to the lack of a monotonicity formula for~$p\neq 2$. 
However, some weak form of monotonicity type formula is known for 
the modified Alt-Caffarelli-Friedman functional, 
namely a discrete monotonicity formula in two spatial dimensions 
when~$p$ is close to~2, see~\cite{DK2}.  

\smallskip 

The aim of this paper is twofold and contributes into the regularity theory of the two-phase 
free boundary problems: first, we define a suitable notion of flatness for free boundary points
which allows to partition the set $\fb u$ into to disjoint subsets $\mathscr  F$ and $\mathscr N$.  
Here $\mathscr F$ is the set of flat free boundary points and $\mathscr N$ the set of non-flat points.
These sets are determined by the critical flatness constant $h_0$, such that if 
the flatness at $x\in \fb u$ is less that $h_0$ then  the free boundary 
must be  regular in some vicinity of $x$. Consequently we can stratify the free boundary points 
and prove linear growth at the non-flat points of free boundary 
(see Section \ref{sec:main} for precise definitions and statements). 
\smallskip

The advantage of this approach is that it avoids using the  
optimal regularity for $u$ everywhere and hence circumvents the obstacle imposed by the lack of 
monotonicity formula. 
However, our technique renders the local Lipschitz continuity 
using a simple consequence of Theorem A below.
Observe that the non-flat points $x\in \mathscr N$ are more 
interesting to study and it is vital to have linear growth at such points $x$ in order to 
classify the  blow-up profiles. 

\smallskip 

Second, to study the flat points  $x\in \mathscr F$ we apply the regularity 
theory developed for viscosity solutions of two-phase free boundary problems. 
To do so we prove  that any local minimizer is also a viscosity solution. 
At flat points we get that the free boundary $\fb u$ is very close to a plane in a suitable 
coordinate system. Consequently, $u$ must be $\e-$monotone with $\e>0$ small, which in turn implies that 
the free boundary is $C^{1, \alpha}$ in some vicinity of $x$. 
This approach, which is based on the 
fusion of variational and viscosity solutions, appears to be new and very useful.  

Finally, from here we conclude the partial regularity of  $\fb u$, 
that is $\fb u$ is countably rectifiable and 
$\mathcal H^{N-1}(\fb u\setminus \fbr u)=0$, 
where $\mathcal H^{N-1}$ is the $(N-1)$-dimensional Hausdorff measure.

\medskip

It is worthwhile to point out that our approach is new even for the classical case $p=2$.

\medskip

In the forthcoming Section \ref{sec:main} we give the precise statements 
of the results that we prove. A detailed plan on the organization of the paper 
will be presented at the end of Section \ref{sec:main}. 

%
%
\section*{Basic Notations}
\begin{tabbing}
$C, C_0, C_n, \cdots$ \hspace{1.55cm}       \=\hbox{generic constants,}\\
$\overline U$       \>\hbox{the closure of a set} $U$,\\
$\partial U$        \>\hbox{the boundary of a set }  $U$,\\
$B_r(x), B_r$ \>\hbox{the ball centered at} $x$ \hbox{with radius} $r>0$, $B_r:=B_r(0)$,  \\
$\Gamma= \fb u$      \> \hbox{the  free boundary} $\partial\{u>0\}$,\\
$\fint$\> mean value integral, \\
$\omega_N$\> the volume of unit ball,\\
$\Omega^+(u):=\{u>0\}$ \> the positivity set of $u$,\\
$\Omega^-(u):=\{u<0\}$ \> the negativity set of $u$,\\
$\mathscr N$\> the set of non-flat free boundary points, see Definition \ref{def-N-set},\\
$\mathscr F$ \> $\fb{u}\setminus \mathscr N$,\\
$\lambda(u)$\>$\lambda_+^p\,\chi_{\{u>0\}} +\lambda_-^p\,\chi_{\{u\le 0\}}$,\\
$\Lambda, \Lambda_0$ \> $\Lambda=\lambda_+^p-\lambda^p_-$, $\Lambda_0=\frac{\Lambda}{p-1}$ the Bernoulli constants.
\end{tabbing}
%
%
\section{Main Results}\label{sec:main}

\subsection{Setup}
The existence of bounded minimizers of the functional in \eqref{Ju} 
can be easily established using the semicontinuity of 
the $p-$Dirichlet energy and the weak convergence in $W^{1,p}$, and can be found in \cite{DP2}. 

Let now $x_0\in \fb u$ and
\begin{equation}
\mathcal S(h; x_0, \nu):=\{x\in \R^n : -h<(x-x_0)\cdot\nu<h\}
\end{equation}
be the slab of height $2h$ in unit direction $\nu$. 
Let $h_{\min}(x_0, r, \nu)$ be the
minimal height of the slab containing the free boundary in $B_r(x_0)$, i.e.
\begin{equation}
h_{\min}(x_0, r, \nu):=\inf\{h : \partial\{u>0\}\cap B_r(x_0)\subset S(h; x_0, \nu)\cap B_r(x_0)\}.
\end{equation}
Put 
\begin{equation}\label{min-h}
h(x_0, r):=\inf_{\nu\in \mathbb S^N}h_{\min}(x_0, r, \nu).
\end{equation}
Clearly $h(x_0, r)$ is {\bf non-decreasing} in $r$.

\begin{AAA}
Let $u$ be a local minimizer of~\eqref{Ju}. 
Then, for any bounded subdomain $D\Subset \Omega$ 
there are positive constants $h_0$ and $L$ depending only 
on $N, p,  \Lambda, \sup_{\Omega}|u|$ and $\dist(\partial\Omega, D)$ such that, 
for any $x_0\in D\cap \partial\{u>0\}$ one of the following two alternatives holds: 
\begin{itemize}
\item if $h(x_0, 2^{-k})\ge h_02^{-k-1}$, for all ~$k\in \mathbb N, 2^{-k}<\dist(\partial\Omega, D)$, then 
$$\sup_{B_{r/2}(x_0)}|u|\leq Lr,$$ 
for all~$0<r<\dist(\partial\Omega,D)$,
\item if $h(x_0, 2^{-k_0})<h_02^{-k_0-1}$ for some $k_0\in \mathbb N$
then the free boundary $\fb u$ is $C^{1, \alpha}$ in some neighbourhood of $x_0$.
\end{itemize}
\end{AAA}

We call $h_0/2$ {\bf the critical flatness constant}. 

The statement in Theorem~A leads to the following definition:

\begin{defn}\label{def-N-set}
We say that $z\in \fb u$ is non-flat if $h(z, 2^{-k})\ge h_02^{-k-1}$ for all $ k\in \mathbb N$ such that $2^{-k}<\dist(z, \partial\Omega)$.
The set of all non-flat points 
is denoted by~$\mathscr N(\Gamma)$ or $\mathscr N$ for short.
\end{defn}

Notice that if $z\not \in \mathscr N$ then $h(z, 2^{-k_0})< h_02^{-k_0-1}$, for some $k_0\in \mathbb N$.
So Theorem A gives a partition of the free boundary of the form 
\begin{equation}\label{lyukiy0909090}
\fb{u}=\mathscr F\cup\mathscr N
\end{equation}
where $\mathscr F:=\left\{x\in \fb{u}: h(x, 2^{-k_0})< h_02^{-k_0-1},\ \text{for some}\ k_0\in \mathbb N\right\}$
is the set of flat free boundary points.

\begin{BBB}
Let $u$ be as in Theorem A. Then, for any subdomain $D\Subset \Omega $  we have
$$\mathcal H^{N-1}(\fb u\cap D)<\infty$$
and 
$$\mathcal H^{N-1} \big( (\fb{ u}\setminus \partial_{\text{red}} \{u>0\}) \cap D\big)=0.$$
In particular, 
$\mathcal H^{N-1}(D \cap \mathscr N)=0$.
\end{BBB}

We remark that, as a consequence of Theorem A, we also 
obtain local Lipschitz continuity for the minimizers.

\begin{CCC}
Let $u$ be as in Theorem A. Then for any subdomain $D\Subset \Omega$ there is a
constant $C_0$ depending only on $\dist(D, \partial \Omega), N, p, \Lambda, h_0$ and $L$ such that  
\begin{equation}\label{ssuka}
|u(x)-u(y)|\leq C_0|x-y|, \quad \forall x, y\in D.
\end{equation}
\end{CCC}
The proof of Theorem C will be given in Section \ref{sec-ssuka}.

\subsection{Strategy of the proofs}
The methods and the techniques that we employ to prove Theorems A and B 
pave the way to a number of new approaches. 

\smallskip 

First, we fuse the variational methods with the viscosity theory. 
This is done by proving that any local 
minimizer $u\in W^{1, p}(\Omega)$ is also a viscosity solution
(see Section \ref{sec:visc}, and in particular Theorem \ref{TH:viscosity}). 
The key ingredient in the proof is the linear development of a nonnegative $p-$harmonic function $v$
in $D\subset \R^N$ near $x_0\in \partial D$ that vanishes continuously on $B_r(x_0)\cap \partial D$, see
Lemma \ref{lemma:linear}. 
There is a subtle point in the proof of the linear development lemma which amounts to the following claim: 
if $x_0\in \fb u$ and $B_r(y_0)\subset \{u>0\}$
with $x_0\in \partial B_r(y_0)$ then $u$ has linear growth near $x_0$, i.e. there is a
constant $C(x_0)>0$ (depending on $x_0$) such that $|u(x)|\leq C(x_0)|x-x_0|$ near $x_0$.
Indeed, by standard barrier argument we have that  
$$u^-(x)\leq \sup\limits_{B_{2r}(y_0)}u^-\frac{\Phi(|x-y_0|)-\Phi(r)}{\Phi(2r)-\Phi(r)}$$ 
where $\Phi(t)=t^{\frac{p-N}{p-1}}$. Therefore $u^-$ has linear growth near $x_0$. Now 
the linear growth of $u^+$ near $x_0$ follows form  Lemma \ref{lem:dndooo}. Clearly the 
same claim is valid if $B_r(y_0)\subset \{u<0\}$ and $x_0\in \partial B_r(y_0)$.
We stress on the fact that Lemma \ref{lemma:linear}  on linear development 
remains valid for solutions to a wider class of equations 
for which Harnack's inequality and Hopf's Lemma are valid. 

\medskip 

Second, we compare $r=2^{-k}$ with the minimal height 
$h(x_0, r)$ of the parallel slab of planes containing $B_r(x_0)\cap \fb u$, for $x_0\in \fb u$. 
More precisely, take $k\in \mathbb N$, and fix $h_0>0$, then  
\begin{equation}\label{rivas-1}
\text{either $h(x_0, 2^{-k})\ge h_0 2^{-k-1}$,} 
\end{equation}
or 
\begin{equation}\label{rivas-2}
\text{$h(x_0, 2^{-k})<h_0 2^{-k-1}$}. 
\end{equation}
Consequently, for given $x_0\in \fb u$ there are two alternatives: either for some  $k$ we arrive at \eqref{rivas-2}
and this will mean that $x_0$ is a flat point of $\fb{u}$ (if $h_0>0$ is small) or \eqref{rivas-1} 
holds for sufficiently large
$k\ge k_0$. The latter implies linear growth at $x_0$. Note that 
the non-flat points are more interesting to study and having the linear growth at 
such points allows one to use compactness argument and blow-up $u$ in order to
study the properties of the resulted configuration as done in the proofs of \eqref{p-lap-nondeg}, \eqref{p-lap-nondeg-0}
and \eqref{neg-dens}.  Note that if \eqref{rivas-1} holds for $1\leq k<k_0$ then we have 
linear growth for $u$ near $x_0$ unto the level $2^{-k_0}$, see Corollary \ref{cor:lin flat}.

Altogether, this approach allows us to 
prove the main properties of the free boundary without using the full optimal regularity of $u$
and can be applied to a wide class of variational free boundary problems with two phases.
A diagram showing the scheme of the proof is given below.

\smallskip 

As for the proof of the partial regularity result, i.e.
$$\mathcal H^{N-1} \left( \fb{ u}\setminus \partial_{\text{red}} \{u>0\}\right)=0,$$
we employ a non-degeneracy result obtained in Proposition \ref{str-nd-hi} for $u^+$ 
and some estimates for the Radon measure 
$\Delta_p u^+$ given in Lemma \ref{lem-7.11}. This is a standard approach but more involved because 
the linear growth is valid only at non-flat points of the free boundary.

\smallskip

\subsection{Structure of the paper}
In Section~\ref{sec:tec} we collect some material, mostly of technical nature, that 
we will use in the other sections. In particular, we prove the continuity of minimizers, 
by showing that $\nabla u\in BMO_{loc}$ if 
$p>2$  and $|\nabla u|^{\frac p2-1}\nabla u \in BMO_{loc}$ if $1<p<2$. We also recall the 
Liouville's Theorem and some basic 
properties of minimizers. Finally we show that $u^+$ is non-degenerate, 
in the sense of Proposition \ref{str-nd-hi}, and a coherence lemma (see Lemma \ref{lem:dndooo}). 

In Section~\ref{sec:visc} we prove that any minimizer of the functional in~\eqref{Ju} 
is also a viscosity solution, according to Definition~\ref{def:visc}. This will allow
us to apply the regularity theory developed in~\cite{LN1, LN2} for viscosity solutions
and infer that the free boundary is $C^{1, \alpha}$ regular near flat points. 

In Section~\ref{sec:eps} we discuss and compare the notions of~$\epsilon$-monotonicity 
of minimizers and of slab flatness of the free boundary. 

Section~\ref{lin vs flat} is devoted to the proof of Theorem~A
and Section~\ref{sec:thb} contains the set up for the proof of Theorem~B. 
In Section~\ref{sec:blow} we deal with the blow-up of minimizers 
proving  some useful convergence and finish the proof of Theorem B. 

Then in Section~\ref{sec-ssuka} we prove Theorem C.

The paper contains also an appendix, where we prove a result 
needed in Section~\ref{sec:visc}. 

\tikzstyle{decision} = [diamond, draw, fill=red!20, 
    text width=7.5em, text badly centered, node distance=3cm, inner sep=0pt]
\tikzstyle{block} = [rectangle, draw, fill=blue!20, 
    text width=11em, text centered, rounded corners, minimum height=4em]
\tikzstyle{line} = [draw, -latex']
\tikzstyle{cloud} = [draw, ellipse,fill=red!20, node distance=3cm,
    minimum height=2em,  text width=11em, text centered]
\tikzstyle{cloud1} = [draw, ellipse,fill=red!20, node distance=3cm,
    minimum height=2em]


\scalebox{.87}{
\begin{tikzpicture}[node distance = 2cm, auto, remember picture]
    \node [decision] (init) {\small {Let $u$ be a minimizer and $x_0\in \partial \{u>0\}$;}};
    \node [block, left of=init, node distance=5cm] (tbplus) {\small if there is a touching ball $B\subset\Omega^+(u)$ such that $x_0\in \partial B$};
    \node [block, below of=tbplus, node distance=2.5cm] (uminlin) {\small then there is a constant $C(x_0)$ depending on $x_0$ such that $\sup\limits_{B_\rho(x_0)}u^-\leq C(x_0)\rho$};
    \node [block, below of=uminlin, node distance=3cm] (upluslin) {\small  from Lemma \ref{lem:dndooo} there is a constant $C(x_0)$ depending on $x_0$ such that $\sup\limits_{B_\rho(x_0)}u^+\leq C(x_0)\rho$};
    \node [block, below of=upluslin, node distance=2.5cm] (asym) {\small from the linear growth at $x_0$ we have the linear development Lemma \ref{lemma:linear}};
    \node [cloud1, below of=asym, node distance=2cm] (stop) {\small then $u$ is a viscosity solution};
     \node [block, right of=init, node distance=5cm] (tbplus1) {\small if there is a touching ball $B\subset\Omega^-(u)$ such that $x_0\in \partial B$};
    \node [block, below of=tbplus1, node distance=2.5cm] (uminlin1) {\small then there is a constant $C(x_0)$ depending on $x_0$ such that $\sup\limits_{B_\rho(x_0)}u^+\leq C(x_0)\rho$};
    \node [block, below of=uminlin1, node distance=3cm] (upluslin1) {\small from Lemma \ref{lem:dndooo}  there is a constant $C(x_0)$ depending on $x_0$ such that $\sup\limits_{B_\rho(x_0)}u^-\leq C(x_0)\rho$};
    \node [block, below of=upluslin1, node distance=2.5cm] (asym1) {\small from the linear growth at $x_0$ we have the linear development Lemma \ref{lemma:linear}};
\node [block, below of=init, node distance=12cm] (step1) {\small for all $k\in  N$ we have $h(x_0, 2^{-k})\ge 2^{-k-1}$};
\node [cloud, right of=step1, node distance=6cm] (linear) {\small then  $\sup\limits_{B_\rho(x_0)}|u|\leq 4L\rho$ with tame $L>0$};
\node [block, below of =step1, node distance=3cm] (existk0) {\small there is $k_0$ such that $h(x_0, 2^{-k_0})<h_02^{-k_0-1}$ hence by Corollary \ref{cor:lin flat} $\sup\limits_{B_{2^{-k_0}}}|u|\le 4L2^{-k_0}$};
\node [block, below of=existk0, node distance=3cm] (scale) {\small then the scaled  function $v(x)=\frac{u(x_0+2^{-k_0}x)}{2^{-k_0}}$, $x\in B_1$, has flatness of the free boundary $h(0, 1)<h_0/2$};
\node [cloud1, below of=scale, node distance=2.5cm] (flat) {\small thus  $\partial \{v>0\}\cap B_{\delta}$ is $ C^{1, \gamma}$ regular with some tame $\delta, \gamma\in(0, 1)$};
    %
    %
    \path [line] (init) -- (tbplus);
    \path [line] (tbplus) -- (uminlin);
    \path [line] (uminlin) -- (upluslin);
    \path [line] (upluslin) -- (asym);
    \path [line] (asym) -- (stop);
   \path [line] (init) -- (tbplus1);
    \path [line] (tbplus1) -- (uminlin1);
    \path [line] (uminlin1) -- (upluslin1);
    \path [line] (upluslin1) -- (asym1);
    \path [line] (asym1) -- (stop);
    \path [line] (step1) -- node {Yes}(linear);
    \path [line] (step1) -- node {No} (existk0);
    \path [line] (stop) |-  (step1);
    \path [line] (existk0) -- (scale);
    \path [line] (scale) -- (flat);
\end{tikzpicture}
}

%
%
\section{Technicalities}\label{sec:tec}
In this section we prove some basic properties of minimizers.

\subsection{A BMO estimate for $\na u$.} \label{sub:coherence}
We first prove the continuity of minimizers of~\eqref{Ju} with
any $\alpha-$H\"older  modulus of continuity, with $\alpha\in (0,1)$, 
if $p\in(1,2)$ and 
log-Lipschitz modulus of continuity if $p>2$. 
Our method is a variation of 
\cite{ACF} and uses some standard inequalities for the functionals with $p-$power growth. 

\begin{lemma}[Continuity of minimizers]\label{lemma:coherence}
Let~$u$ be a minimizer of~\eqref{Ju}. Then 
\begin{itemize}
\item for $1<p<2$, we have that 
$|\na u|^{\frac{p-2}2}\na u \in BMO(D)$ for any bounded subdomain $D\Subset  \Omega$, 
and consequently $u\in C^\sigma(D)$ for any $\sigma\in (0, 1)$,
\item for $p>2$, we have that $\na u \in BMO(D)$, 
for any bounded subdomain $D\Subset\Omega$,
and thus $u$ is locally log-Lipschitz continuous.
\end{itemize}
In particular, $\na u\in L^q(D)$ for any $1<q<\infty$ and for any $p>1$.
\end{lemma}

\begin{proof} Fix~$R\ge r>0$ and $x_0\in D$ such that $B_{2R(x_0)}\Subset D$.
Let $v$ be the solution of 
$$ \left\{\begin{array}{ll}
\Delta_p v=0 & {\mbox{ in }} B_{2R}(x_0),\\
v=u & {\mbox{ on }} \partial B_{2R}(x_0). 
\end{array}
\right.$$
Comparing $J(u)$ with $J(v)$ in $B_{2R}(x_0)$ yields 
\begin{equation}\begin{split}\label{Donatel-1}
\int_{B_{2R}(x_0)} |\na u|^p-|\nabla v|^p\, 
\le&\int_{B_{2R}(x_0)}\lambda^p_+\chi_{\{v>0\}}+\lambda^p_-\chi_{\{v\le 0\}}
-\left(\lambda^p_+\chi_{\{u>0\}}+\lambda^p_-\chi_{\{u\le 0\}}\right)\\
&\,\leq CR^N,
\end{split}
\end{equation}
for some $C>0$. On the other hand, the following estimate is true  (see \cite{DP2} page 100)
\begin{equation}\label{Donatel-2}
\int_{B_{2R}(x_0)}|\na u|^p-|\nabla v|^p\ge
\gamma \left\{
\begin{array}{lll}
\int_{B_{2R}(x_0)}(|\na u|+|\na v|)^{p-2}|\na(u-v)|^2, & {\mbox{ if }}1<p<2,\\
\int_{B_{2R}(x_0)}|\na(u-v)|^p, &{\mbox{ if }}p>2,
\end{array}
\right.
\end{equation}
for some tame constant $\gamma>0$ depending on $N$ and $p$.

Introduce the function $V:\R^N\to \R^N$ defined as follows 
\begin{equation}\label{Donatel-3}
V(\xi):=\left\{
\begin{array}{lll}
|\xi|^{\frac{p-2}2}\xi, & {\mbox{ if }}1<p<2, \\
\xi, &{\mbox{ if }} p>2,
\end{array}
\right.
\end{equation}
then from the basic inequalities 
\begin{equation}\label{Duzaa-inq}
c^{-1}(|\xi|^2+|\eta|^2)^{\frac{p-2}2}|\xi-\eta|^2\leq |V(\xi)-V(\eta)|^2\leq c(|\xi|^2+|\eta|^2)^{\frac{p-2}2}|\xi-\eta|^2,
\end{equation}
that are valid for any $p>1$
(see \cite{Duzaar} page 240),
we infer 
the estimate 
\begin{eqnarray}\label{est-RN-p}
\int_{B_{2R}(x_0)}|V(\na u)-V(\na v)|^2\leq CR^N,
\end{eqnarray}
up to renaming~$C$. Indeed, the case $1<p<2$ follows from the second 
inequality in \eqref{Duzaa-inq}. As for the remaining 
case $p>2$ we have by H\"older's inequality 
$$\left(\fint_{B_{2R}(x_0)}|\na u-\na v|^p\right)^{\frac1p}\geq 
\left(\fint_{B_{2R}(x_0)}|\na u-\na v|^2\right)^{\frac12}$$
and \eqref{est-RN-p} follows.

Furthermore, for any~$\rho>0$, we set 
$$ (V(\na u))_{x_0,\rho}:=\fint_{B_\rho(x_0)}V(\nabla u).$$
Then, from  H\"older's inequality we have 
\begin{equation}\begin{split}\label{starqew}
|(V(\na v))_{x_0,r}-(V(\na u))_{x_0,r}|^2 \le\,&
\left(\fint_{B_r(x_0)}|V(\nabla v)-V(\nabla u)|\right)^2\\
\le\,& \fint_{B_r(x_0)}|V(\nabla v)-V(\nabla u)|^2.
\end{split}\end{equation}

We would also need the following estimate for a $p-$harmonic function $v$:
there is $\alpha>0$ such that for all balls $B_{2R(x_0)}\Subset D$, 
with~$R\ge r>0$, there exists a universal constant $c>0$ 
such that the following Campanato type estimate is valid
\begin{equation}\label{Companato-00}
\fint_{B_r(x_0)}|V(\na v)-V((\na v)_{x_0, r})|^2\leq c \left(\frac rR\right)^{\alpha}
\fint_{B_R(x_0)}|V(\na v)-V((\na v)_{x_0, R})|^2.
\end{equation}
See \cite{Diening} Theorem 6.4 for $V(\na v)=|\na v|^{\frac{p-2}2}\na v$ and \cite{DiB-M} Theorem 5.1
for $V(\na v)=\na v$. 

Denote $\|\cdot \|_{L^2({B_r(x_0)})}=\|\cdot \|_{2,r}$, then, 
using~\eqref{starqew}, we obtain 
\begin{eqnarray}\nonumber
\|V(\na u)-(V(\na u))_{x_0,r}\|_{2,r}&\le& \|V(\na u)-V(\na v)\|_{2,r}+\|V(\na v)-(V(\na v))_{x_0,r}\|_{2,r}\\\nonumber&&
+\,\|(V(\na v))_{x_0,r}-(V(\na u))_{x_0,r}\|_{2,r}\\\nonumber
&\le& 2\,\|V(\na u)-V(\na v)\|_{2,r}+\|V(\na v)-(V(\na v))_{x_0,r}\|_{2,r}\\\label{Campanato}
&\le& 2\,\|V(\na u)-V(\na v)\|_{2,r}+C\left(\frac rR\right)^{\frac{N+\alpha}2}\|V(\na v)-(V(\na v))_{x_0,R}\|_{2,R},\\\nonumber
\end{eqnarray}
where, in order to get \eqref{Campanato}, we used Campanato type estimate \eqref{Companato-00}.

From the triangle inequality for $L^2$ norm  we have
\begin{eqnarray*}
\|V(\na v)-(V(\na v))_{x_0,R}\|_{2,R}\leq 2\,\|V(\na u)-V(\na v)\|_{2,R} +\|V(\na u)-(V(\na u))_{x_0,R}\|_{2,R},
\end{eqnarray*}
and so, combining this with \eqref{est-RN-p}, we obtain 
\begin{eqnarray*}
\|V(\na u)-(V(\na u))_{x_0,r}\|_{2,r}&\le& 2\,\|V(\na u)-V(\na v)\|_{2,r}\\\nonumber&&
+C\left(\frac rR\right)^{\frac{N+\alpha}2}\big[2\|V(\na u)-V(\na v)\|_{2,R} +\|V(\na u)-(V(\na u))_{x_0,R}\|_{2,R}\big]\\\nonumber
&\le& C\left\{\|V(\na u)-V(\na v)\|_{2, R}+
\left(\frac rR\right)^{\frac{N+\alpha}2}
 \|V(\na u)-(V(\na u))_{x_0,R}\|_{2, R}\right\}\\\nonumber
&\le& A\left(\frac rR\right)^{\frac{N+\alpha}2}
 \|V(\na u)-V((\na u))_{x_0,R}\|_{2, R}+B R^{\frac{N}2},\\\nonumber
\end{eqnarray*}
for some tame positive constants $A$ and $B$.

Introduce 
$$ \phi(r):=\sup\limits_{t\leq r} \|V(\na u)-(V(\na u))_{x_0,t}\|_{2,t},
$$ 
then the former inequality can be rewritten as 
$$\phi(r)\le A\left(\frac rR\right)^{\frac{N+\alpha}2}\phi(R)+B R^{\frac N2},$$
with some positive constants $A, B, \alpha$. Applying
Lemma 2.1 from \cite{Giaq} Chapter~3, we  conclude that 
there exist~$R_0>0$ and~$c>0$ such that
$$\phi(r)\leq cr^{\frac N2}\left(\frac{\phi(R)}{R^{\frac N2}}+B\right),$$
for all $r\le R\le R_0$,
and hence
$$ \int_{B_r(x_0)}|V(\na u)-(V(\na u))_{x_0,r}|^2 \le C r^N, $$
for some tame constant $C>0$.
This shows that~$V(\nabla u)$ is locally BMO.
The log-Lipschitz estimate for $p>2$ now follows from 
\cite{Cianchi} Theorem 3. The H\"older continuity follows from 
Sobolev's embedding  and the John-Nirenberg Lemma. 
\end{proof}

\begin{remark}
From Lemma \ref{lemma:coherence} it follows that for any $D\Subset \Omega$ 
there is a constant $C>0$ depending only on $N, p, \Lambda, \sup_{\Omega}|u|$ and $\dist(D, \partial \Omega)$
such that if $p>2$ and $x_0\in \Gamma$,
then
$$ \left|\fint_{\partial B_r(x_0)}u\right|\le Cr \quad {\mbox{ for any }} 
B_r(x_0)\subset D,$$ 
see \cite{DK2}.
\end{remark}

\subsection{Liouville's Theorem}\label{app:liouville}

This section is devoted to Liouville's Theorem, 
that we use in the proof of Proposition~\ref{prop:lin flat}. 
We add the proof here.

\begin{theorem}\label{th:liouville}
Let~$U$ be a~$p$-harmonic function in~$\R^N$ such that 
\begin{equation}\label{pgrejdgkerh}
|U(x)|\le C |x|, \quad {\mbox{ for any }} x\in\R^N,
\end{equation}
for some~$C>0$. 
Then~$U$ is a linear function in~$\R^N$. 
\end{theorem}

\begin{proof} 
For any~$r>0$, we introduce the scaled function
\begin{equation}\label{uerre}
U_r(x):=\frac{U(rx)}{r}.
\end{equation}
Hence~$U_r$ is a~$p$-harmonic function and 
$$ |U_r(x)|\le C |x|, \quad {\mbox{ for any }} x\in\R^N,$$
thanks to~\eqref{pgrejdgkerh}. 

Moreover, from the~$C^{1,\alpha}$-estimates for~$p$-harmonic functions in~$B_1$ 
(for some~$\alpha\in(0,1)$), see~\cite{T1}, we have that  $\sup_{B_1}|\nabla U_r(x)|\le M$
and, moreover, 
\begin{equation}\label{pewgbnazsdftht}
\frac{|\nabla U_r(x)-\nabla U_r(y)|}{|x-y|^\alpha}=\frac{|\nabla U(rx)-\nabla U(ry)|}{|x-y|^\alpha}\le M, \quad x, y\in B_1, x\not=y
\end{equation}
for a positive constant~$M$, 
depending only  on~$N$, $p$ and~$\sup_{B_2}|U_r(x)|\le2C$. 

Hence, taking~$\xi:=rx$ and~$\eta:=ry$ in~\eqref{pewgbnazsdftht}, we obtain that 
for any~$r>0$
\begin{equation}\label{pewgbnazsdftht-1}
|\nabla U(\xi)-\nabla U(\eta)|\le \frac{M}{r^\alpha}|\xi-\eta|^\alpha, \quad 
{\mbox{ for any }} \xi,\eta \in B_r.
\end{equation}
In particular, \eqref{pewgbnazsdftht-1} holds true for any~$r>1$. 
Therefore,  letting ~$\xi,\eta\in B_1$ and
sending~$r\to+\infty$ in the formula above, we obtain that
$$ |\nabla U(\xi)-\nabla U(\eta)|=0 \quad {\mbox{ for any }} \xi,\eta\in B_1. $$
Hence, $U$ is linear in~$B_1$. This completes the proof in view of the 
Unique Continuation Theorem \cite{GM}. 
\end{proof}

\subsection{Some basic properties of the local minimizers of $J$}

\begin{prop}\label{prop:tec}
Let $u\in W^{1,p}$ be a local minimizer of~\eqref{Ju}. Then
\begin{itemize}
\item[P.1] $\Delta_p u^\pm\geq 0$ in the sense of distributions and $\Delta_p u=0$ in $\{u>0\}\cup \{u<0\}$,
\item[P.2] for any $D\Subset \Omega$ there is $c_0>0$ depending only on 
$N, p, \Lambda,\sup_\Omega |u|$ and $\dist(D, \partial \Omega)$ such that if
$$\limsup_{r\to 0}\frac{|B_r(x_0)\cap \{u<0\} |}{|B_r(x_0)|}\le c_0, \quad x_0\in \Gamma\cap D$$
then $\sup_{B_r(x_0)}|u|\leq \frac C{c_0}r$ where $C$ is a tame constant. 
\end{itemize}
\end{prop}

\begin{proof}
P.1 follows from a standard comparison of~$u$ and~$u+\epsilon\varphi$, 
where~$\varphi$ is a suitable smooth and compactly supported function. 
P.2 follows from \cite{K1}.
\end{proof}

\subsection{A remark on the volume term and scaling}\label{rem:one phase}
It is convenient to define
\begin{equation}\label{lambda u} 
\lambda(u):=\lambda^p_+\chi_{\{u>0\}}+\lambda^p_-\chi_{\{u\le 0\}}
= \Lambda \chi_{\{u>0\}} +\lambda^p_-,
\end{equation}
with~$\Lambda:=\lambda^p_+-\lambda^p_->0$. 
As a consequence,  the functional in~\eqref{Ju} can be rewritten in an equivalent form 
\begin{equation}\label{J-smpl} 
J(u)= \int_{\Omega}|\nabla u|^p +\Lambda\chi_{\{u>0\}}
+\lambda^p_-|\Omega|.
\end{equation}
Notice that the last term does not affect the minimization problem,
and so if~$u$ is a minimizer for~$J$, then it is also a minimizer for
\begin{equation}\label{erdfgvxb}
\tilde{J}(u):= \int_{\Omega}|\nabla u|^p +\Lambda\chi_{\{u>0\}}.
\end{equation}
Observe that if $\Lambda>0$ then the free boundary~$\partial\{u>0\}\cup\partial\{u< 0\}$
for the minimizer~$u$ of~$J$ coincides with~$\partial\{u>0\}$.
Indeed, let $\Gamma_0:=\partial\{u< 0\}\setminus \partial\{u> 0\}$, then we clearly have 
that if $x_0\in \Gamma_0$ then there is $r>0$ such that $u\le 0$ in 
$B_r(x_0)$, and so $u$ is $p-$superharmonic in $B_r(x_0)$. 
On the other hand, we have that $\Lambda=\lambda_+^p-\lambda_-^p>0$, 
and so we get a contradiction with P.1 of Proposition \ref{prop:tec}.
Therefore $\Gamma_0=\emptyset$.
\medskip 

The functional $\tilde J$ preserves the minimizers under certain scaling. 
This property is a key ingredient in a number of arguments to follow. 

More precisely, let $u$ be a minimizer of \eqref{Ju}, 
and take $x_0\in \fb u$ and $r>0$ such that $B_r(x_0)\subset \Omega$. 
Fixed $\rho>0$, set also $u_\rho(x):=\frac{u(x_0+\rho x)}S$, for some constant $S>0$. Then 
one can readily verify that
\begin{eqnarray}\label{gen-qash}
\int_{B_1}|\na u_\rho(x)|^p+\left[\frac{\rho}{S}\right]^p\Lambda\X{u_\rho>0}=
\left[\frac{\rho}{S}\right]^p\frac1{\rho^N}\int_{B_\rho(x_0)}|\na u|^p+\Lambda\X{u>0}.
\end{eqnarray}
In particular if we let $S=\rho$ then 
\begin{eqnarray}\label{mek-qash}
\int_{B_1}|\na u_\rho(x)|^p+\Lambda\X{u_\rho>0}=
\frac1{\rho^N}\int_{B_\rho(x_0)}|\na u|^p+\Lambda\X{u>0}.
\end{eqnarray}
Therefore if $u$ is minimizer of $\tilde J$ in $B_\rho(x_0)$ then 
the scaled function $u_\rho$ is a minimizer of $\tilde J$ in $B_1$.

\subsection{Strong Non-degeneracy}
In this section we deal with a strong form of non-degeneracy for 
minimizers of \eqref{Ju}. 
For $p=2$, this result is contained in~\cite{ACF} (see in particular Theorem~3.1 there). 
We use a modification of an argument from \cite{ACF-quasi} Lemma 2.5.

\begin{prop}\label{str-nd-hi}
For any $\kappa\in(0,1)$ there exists a constant $c_\kappa>0$ such that 
for any local minimizer of~\eqref{Ju} and for any small ball $B_r\subset \Omega$
\begin{equation}
\mbox{if}\quad \frac1r\left(\fint_{B_r}(u^+)^p\right)^{\frac1p}<c_\kappa\ \mbox{then}\  u\equiv 0\  \mbox{in}\  B_{\kappa r}.
\end{equation}
\end{prop}

\begin{proof}
By scale invariance of the problem we take $r=1$ for simplicity
and put 
\begin{equation}\label{3.14bis}
\e:=\frac1{\sqrt{\kappa}}\sup_{B_{\sqrt{\kappa}}}u^+.\end{equation}
Since $u^+$ is $p-$subharmonic (recall~P.1 in Proposition~\ref{prop:tec}), 
then by \cite{MZ} Theorem 3.9
\begin{equation*}
\e\leq \frac1{\sqrt\kappa}\frac{C(p, N)}{(1-\sqrt\kappa)^{\frac Np}} \left(\fint_{B_1}(u^+)^p\right)^{\frac1p}.\end{equation*} 
Introduce 
\begin{equation*}
v(x):=\left\{
\begin{array}{ll}
C_1\e\left[e^{-\mu |x|^2}-e^{-\mu\kappa^2}\right] & \text{in}\ B_{\sqrt\kappa}\setminus B_{\kappa},\\
0 & \text{in}\ B_{\kappa},
\end{array}
\right.
\end{equation*}
where~$\mu>0$ and~$C_1$ is chosen so that 
\begin{equation}\label{test-bndr-nd}
v|_{\partial B_{\sqrt\kappa}}:=\sqrt\kappa\e=\sup_{B_{\sqrt\kappa}}u^+\ge u|_{\partial B_{\sqrt\kappa}}, 
\end{equation}
that is
$$C_1=\frac{\sqrt\kappa}{e^{-\mu\kappa}-e^{-\mu\kappa^2}}.$$

Furthermore, by a direct computation we can see that 
\begin{equation}\label{starter}
\na v=-C_1\e\,2\mu xe^{-\mu|x|^2} \quad {\mbox{ in }} B_{\sqrt{\kappa}}\setminus B_{\kappa},
\end{equation}
and 
$$\Delta_p v(x)=C_1\e \,(p-1)(2\mu)^2|\nabla v|^{p-2}e^{-\mu|x|^2}\left(|x|^2-\frac{N+p-2}{2\mu(p-1)}\right),$$
see~\cite{Aram JDE}. 
Thus 
\begin{equation}\label{display123}
{\mbox{$v$ is $p-$superharmonic in $B_{\sqrt\kappa}\setminus B_{\kappa}$}}
\end{equation} 
if $\mu$ is sufficiently small, say, 
$$\mu<\frac{N+p-2}{2\kappa(p-1)}.$$
It is clear that $\min\{u, v\}=u$ on $\partial B_{\sqrt\kappa}$, 
thanks to~\eqref{test-bndr-nd}, hence by the minimality of~$u$ 
(recall also Subsection~\ref{rem:one phase})
\begin{equation}\label{starstar} 
\tilde J(u)\leq \tilde J(\min\{u, v\}). \end{equation} 
Now we observe that 
\begin{eqnarray*}
\tilde J(\min\{u, v\})&=&\int_{B_\kappa}|\na \min\{u, v\}|^p+\Lambda\X{\min\{u, v\}>0}\\\nonumber
&&+
\int_{B_{\sqrt\kappa}\setminus B_\kappa}|\na \min\{u, v\}|^p+\Lambda\X{\min\{u, v\}>0}\\\nonumber
&=&\int_{B_\kappa\cap \{u\le 0\}}|\na u|^p+\Lambda\X{u>0}\\\nonumber
&&+
\int_{B_{\sqrt\kappa}\setminus B_\kappa}|\na \min\{u, v\}|^p+\Lambda\X{\min\{u, v\}>0},
\end{eqnarray*}
while 
\begin{eqnarray*}
\tilde J(u)&=&\int_{B_\kappa\cap \{u\le 0\}}|\na u|^p+\Lambda\X{u>0}\\\nonumber
&&+\int_{B_\kappa\cap \{u> 0\}}|\na u|^p+\Lambda\X{u>0}+
\int_{B_{\sqrt\kappa}\setminus B_\kappa}|\na u|^p+\Lambda\X{u>0}.
\end{eqnarray*}
\allowdisplaybreaks
Therefore, from~\eqref{starstar}, we have that 
\begin{eqnarray*}
\int_{B_\kappa\cap \{u>0\}}|\na u|^p+\Lambda\X{u>0}&\leq& \int_{B_{\sqrt\kappa}\setminus B_\kappa}|\na \min\{u, v\}|^p+\Lambda\X{\min\{u, v\}>0}\\\nonumber
&&-\int_{B_{\sqrt\kappa}\setminus B_\kappa}|\na u|^p+\Lambda\X{u>0}\\\nonumber
&\le & \int_{B_{\sqrt\kappa}\setminus B_\kappa}|\na \min\{u, v\}|^p-|\na u|^p\\\nonumber
&=& \int_{(B_{\sqrt\kappa}\setminus B_\kappa)\cap \{u>v\} }|\na v|^p-|\na u|^p\\\nonumber
&\le&-p\int_{B_{\sqrt\kappa}\setminus B_\kappa}|\na v|^{p-2}\na v \cdot \na \max\{u-v, 0\}\\\nonumber
&=&-p\int_{B_{\sqrt\kappa}\setminus B_\kappa}-\Delta_p v\max\{u-v, 0\}+\mbox{div}(|\na v|^{p-2}\na v\max\{u-v, 0\})\\\nonumber
&\le & -p\int_{B_{\sqrt\kappa}\setminus B_\kappa} \mbox{div} (|\na v|^{p-2}\na v\max\{u-v, 0\})\\\nonumber
&=& p\int_{\partial B_{\kappa}} |\na v|^{p-2}\na v\cdot \nu
\max\{u-v, 0\}\\\nonumber
&=& p\int_{\partial B_{\kappa}} |\na v|^{p-2}(\na v\cdot \nu) u^+,\nonumber
\end{eqnarray*}
where to get the last line we also used the fact that $v$ is a $p-$supersolution 
in~$B_{\sqrt{\kappa}}\setminus B_\kappa$ (recall~\eqref{display123}) and~\eqref{test-bndr-nd}). 
Moreover, by~\eqref{starter}, we have that 
$|\na v|=C_1\e2\mu\kappa e^{-\mu\kappa^2}\leq C\e$ 
on $\partial B_{\kappa}$, for some~$C>0$. Thus 
\begin{equation}\label{3.18bis}
\int_{B_\kappa\cap \{u> 0\}}|\na u|^p+\Lambda\X{u>0}\leq p(C\e)^{p-1}\int_{\partial B_\kappa} u^+.\end{equation}

On the other hand, from trace estimate, Young's inequality and~\eqref{3.14bis}, 
we get 
\begin{equation}\begin{split}\label{3.18ter}
\int_{\partial B_\kappa} u^+ \leq\,& C(N, \kappa)\left(\int_{B_\kappa}u^++\int_{B_\kappa}|\na u^+|\right)
\\\leq\,& C(N, \kappa)\left(\sup_{B_\kappa}u^+\int_{B_\kappa}\X{u>0}+\int_{B_\kappa}\frac1p|\na u^+|^p+\frac1{p'}\X{u>0}\right)\\ 
\le\, &  C(N, \kappa)\left((\e\sqrt\kappa+\frac1{p'})\int_{B_\kappa}\X{u>0}+\frac1p\int_{B_\kappa}|\na u^+|^p\right)\\
\le\, & C_0 \int_{B_\kappa\cap \{u> 0\}}|\na u|^p+\Lambda\X{u>0},
\end{split}\end{equation}
where $p'$ is the conjugate of $p$ and 
$$C_0:=C(N,\kappa)\left(\frac{\e\sqrt\kappa+1/p'}{\Lambda}+\frac1p\right).$$
Thereby, putting together~\eqref{3.18bis} and~\eqref{3.18ter}, we obtain
$$\int_{B_\kappa\cap \{u>0\}}|\na u|^p+\Lambda\X{u>0}\leq p(C\e)^{p-1}\,C_0\,\int_{B_\kappa\cap \{u> 0\}}|\na u|^p+\Lambda\X{u>0},$$
which implies that $u\equiv0$ in $B_{\kappa}$ if $\e$ is small enough.
\end{proof}

As a consequence of Proposition~\ref{str-nd-hi} we have: 

\begin{corollary}\label{cor-nondeg}
Let $u$ be as in Proposition \ref{str-nd-hi}. Let $x\in \fb u$ and $r>0$ such that
$B_r(x)\subset\Omega$. Then 
$$\fint_{B_r(x)}(u^+)^p\ge cr, $$
where $c$ depends only on $\Lambda=\lambda_+^p-\lambda_-^p>0$.
\end{corollary}

\subsection{One phase control implies linear growth}
The last technical estimate is very weak and of pointwise nature.
It is used in the proof of Theorem \ref{TH:viscosity} and serves a preliminary step towards 
the proof of Theorem A. 

\begin{lemma}\label{lem:dndooo}
Let $u$ be a bounded local minimizer of \eqref{Ju}. 
Let $x_0\in \fb u$ and $r>0$ small such that $B_r(x_0)\subset \Omega$. 
Assume that $\sup_{B_r(x_0)}u^-\leq C_0 r$ (resp. $\sup_{B_r(x_0)}u^+\leq C_0 r$), 
for some constant $C_0$ depending on $x_0$.

Then there exists a constant $\sigma>0$ such that $\sup_{B_r(x_0)}u^+\leq \sigma C_0 r$  (resp. 
$\sup_{B_r(x_0)}u^-\leq \sigma C_0 r$).
\end{lemma}

\begin{proof}
We will show only one of the claims, the other can be proved analogously.
Suppose that 
\begin{equation}\label{cx-0-hav}
\sup_{B_r(x_0)}u^-\leq C_0 r
\end{equation} 
and we claim that 
\begin{equation}\label{dndooo}
S(k+1)\le \max  \left\{ \frac{ \sigma C_0}{2^{k+1}}, \frac12 S(k) \right\},
\end{equation} 
where $S(k):=\sup_{B_{2^{-k}}(x_0)}|u|$, for any $k\in\N$. 
To prove this, we argue by contradiction and we
suppose that \eqref{dndooo} fails. Then 
there is a sequence of integers $k_j$, with $j=1,2,\ldots$, such that 
\begin{equation}\label{not-dndooo}
S(k_j+1)> \max  \left\{\frac{ j}{2^{k_j+1}}, \frac12 S(k_j) \right\}.
\end{equation} 
Observe that since $u$ is a bounded minimizer, then \eqref{not-dndooo} implies that 
$k_j\to \infty$ as $j\to+\infty$. 
Also, notice that \eqref{not-dndooo} implies that 
\begin{equation}\label{3.30bis}
\frac{2^{-k_j}}{S(k_j+1)}\le\frac{2}{j}\to 0 \quad {\mbox{ as }} j\to+\infty.
\end{equation}

Now, we introduce the scaled functions $v_j(x):=\frac{u(x_0+2^{-k_j}x)}{S(k_j+1)}$, for any $x\in B_1$.
Then, from \eqref{cx-0-hav} and \eqref{3.30bis}, it follows that 
\begin{equation}\label{dndooo-3}
v_{j}(0)=0 \quad {\mbox{ and }} \quad 
v_j^-(x)=\frac{u^-(x_0+2^{-k_j}x)}{S(k_j+1)}\leq \frac{2^{-k_j} C_0}{S(k_j+1)}<\frac{2C_0}{j}\to 0 
\; {\mbox{ as }} j\to+\infty.
\end{equation}
Also, by \eqref{gen-qash} (used here with $\rho:=2^{-k_j}$ and $S:=S(k_j+1)$) 
we see that $v_j$ is a minimizer of 
the functional 
$$ \int_{B_1}|\na v_j(x)|^p+\left[\frac{2^{-k_j-1}}{S(k_j+1)}\right]^p\Lambda\X{v_j>0}.$$
Furthermore, it is not difficult to see that \eqref{not-dndooo} implies that 
\begin{equation}\label{dndooo-1}
\sup_{B_1}|v_j|\leq 2, \quad {\mbox{ and }} \quad \sup_{B_{\frac12}}|v_j|=1.
\end{equation}
Using this and Caccioppoli's inequality, we infer that 
$$\int_{B_{\frac34}}|\na v_j^\pm|^p\leq 4^pC(N)\int_{B_1} (v_j^\pm)^p\leq 2^{3p}C(N),$$
for some $C(N)>0$, implying that $\|v_j\|_{W^{1,p}(B_{\frac34})}$ are uniformly bounded. 
So using Lemma \ref{lemma:coherence} we can extract 
a converging subsequence such that $v_j\to v_0$ uniformly in $\overline{B_{\frac34}}$ 
and $\na v_j\to \na v_0$ in $L^q(B_{\frac34})$ for any $q>1$. 
Moreover, by \eqref{3.30bis}, 
\begin{equation*}
\int_{B_{\frac34}}|\na v_j(x)|^p+\left[\frac{2^{-k_j-1}}{S(k_j+1)}\right]^p\Lambda\X{v_j>0}\to 
\int_{B_{\frac34}}|\na v_0(x)|^p, \quad {\mbox{ as }} j\to+\infty.
\end{equation*}
This, \eqref{dndooo-3} and \eqref{dndooo-1} give that
\begin{eqnarray*}
\Delta_p v_0(x)=0, \quad v_0(x)\ge 0\  \text{if}\ x\in B_{\frac34}, \quad  v_0(0)=0, 
\quad {\mbox{ and }} \sup_{B_{\frac12}}v_0=1
\end{eqnarray*}
which is in contradiction with the strong minimum principle. 
This shows \eqref{dndooo} and finishes the proof. 
\end{proof}

%
%
\section{Viscosity solutions}\label{sec:visc}

In order to exploit the regularity theory of free boundary developed for the viscosity solutions in \cite{LN1, LN2}
we shall prove that any $W^{1,p}$ minimizer of $J$ is also viscosity solution,
as opposed to Definition~2.4 in~\cite{Luis}. For this, we recall that
$\Omega^+(u)=\{u>0\}$ and~$\Omega^-(u)=\{u<0\}$.
Moreover, if the free boundary is $C^1$  smooth then 
\begin{equation}\label{FB-cond-G}
G(u^+_\nu,u^-_\nu):=(u_\nu^+)^p-(u_\nu^-)^p-\Lambda_0
\end{equation} 
is the flux balance across the free boundary,
where~$u^+_\nu$ and~$u^-_\nu$ are the normal derivatives in the inward direction
to~$\partial \Omega^+(u)$ and~$\partial \Omega^-(u)$, respectively 
(recall that~$\Lambda_0=\frac{\Lambda}{p-1}=\frac{ \lambda^p_+-\lambda^p_-}{p-1}$ is the Bernoulli constant). 

\begin{defn}\label{def:visc}
Let~$\Omega$ be a bounded domain of~$\R^N$ and let~$u$ 
be a continuous function in~$\Omega$. We say that~$u$ is a viscosity solution
in~$\Omega$ if
\begin{itemize}
\item[i)] $\Delta_p u=0$ in~$\Omega^+(u)$ and~$\Omega^-(u)$,
\item[ii)] along the free boundary~$\Gamma$, $u$ satisfies the free boundary condition, in the sense that:
\begin{itemize}
\item[a)] if at~$x_0\in\Gamma$ there exists a ball~$B\subset\Omega^+(u)$
such that~$x_0\in \partial B$ and
\begin{equation}\label{visc1}
u^+(x)\ge\alpha\langle x-x_0,\nu\rangle^+ + o(|x-x_0|), \ {\mbox{ for }} x\in B,
\end{equation}
\begin{equation}\label{visc2}
u^-(x)\le\beta\langle x-x_0,\nu\rangle^- + o(|x-x_0|), \ {\mbox{ for }} x\in B^c,
\end{equation}
for some $\alpha>0$ and~$\beta\ge0$, with equality along every non-tangential domain,
then the free boundary condition is satisfied
$$ G(\alpha,\beta)=0, $$
\item[b)] if at~$x_0\in\Gamma$ there exists a ball~$B\subset\Omega^-(u)$
such that~$x_0\in \partial B$ and
$$ u^-(x)\ge\beta\langle x-x_0,\nu\rangle^- + o(|x-x_0|), \ {\mbox{ for }} x\in B, $$
$$ u^+(x)\le\alpha\langle x-x_0,\nu\rangle^+ + o(|x-x_0|), \ {\mbox{ for }} x\in\partial B, $$
for some $\alpha\ge0$ and~$\beta>0$, with equality along every non-tangential domain,
then
$$ G(\alpha,\beta)=0. $$
\end{itemize}
\end{itemize}
\end{defn}

The main result of this section is the following:
\begin{theorem}\label{TH:viscosity}
Let~$u\in W^{1,p}(\Omega)$ be a minimizer of~\eqref{Ju}.
Then~$u$ is a viscosity solution in~$\Omega$ in the sense of 
Definition~\ref{def:visc}.
\end{theorem}

The proof of  Theorem~\ref{TH:viscosity}, will follow from  Lemma \ref{lemma:linear} below. 
It is a generalization of Lemma~11.17 in~\cite{Luis}
to any~$p$ (see also the appendix in~\cite{DPS}, where the authors deal with
the one-phase problem in the half ball.)
We postpone the proof of Lemma \ref{lemma:linear} to Appendix~\ref{ap:lemma}.

\begin{lemma}\label{lemma:linear}
Let~$0\le u\in W^{1, p}(\Omega)$ be a solution of $\Delta_p u=0$ in~$\Omega$
and~$x_0\in\partial\Omega$. Suppose that~$u$ continuously 
vanishes on~$\partial\Omega\cap B_1(x_0)$.
Then
\begin{itemize}
\item[a)] if there exists a ball~$B\subset\Omega$ touching~$\partial\Omega$ at~$x_0$,
then either~$u$ grows faster than any linear function at~$x_0$, or there exists
a constant~$\alpha>0$ such that
\begin{equation}\label{linear alpha}
u(x)\ge \alpha\langle x-x_0,\nu\rangle^+ +o(|x-x_0|) \quad {\mbox{ in }}B,
\end{equation}
where~$\nu$ is the unit normal to~$\partial B$ at~$x_0$, inward to~$\Omega$.
Moreover, equality holds in~\eqref{linear alpha} in any non-tangential domain.
\item[b)] if there exists a ball~$B\subset\Omega^c$ touching~$\partial\Omega$ at~$x_0$,
then there exists a constant~$\beta\ge0$ such that
\begin{equation}\label{linear beta}
u(x)\le \beta\langle x-x_0,\nu\rangle^+ +o(|x-x_0|) \quad {\mbox{ in }}B^c,
\end{equation}
with equality in any non-tangential domain.
\end{itemize}
\end{lemma}

With this, we are able to prove Theorem~\ref{TH:viscosity}.

\begin{proof}[Proof of Theorem~\ref{TH:viscosity}]
First we observe that~i) in Definition~\ref{def:visc}
is satisfied, thanks to~P.1 in Proposition~\ref{prop:tec}.

To prove~ii), we let ~$x_0\in\Gamma\cap B$, 
$B\subset\{u>0\}$ be a ball touching~$\Gamma$ at~$x_0$ 
and~$\nu$ be the unit vector at~$x_0$
pointing to the centre of~$B$. We want to show that ~\eqref{visc1} and~\eqref{visc2}
are satisfied for some~$\alpha>0$ and~$\beta\ge0$,
with equality in every non-tangential domain.

Notice that $\beta$ is finite, thanks to Lemma~\ref{lemma:linear}
(in particular, the statement~b) applied to~$u^-$).
This follows from a standard barrier argument as one compares 
$u^-$ with $$b(x)=\sup\limits_{B_{2r}(y_0)}u^-\frac{\Phi(|x-y_0|)-\Phi(r)}{\Phi(2r)-\Phi(r)}, \quad x\in B_{2r}(y_0)\setminus B_r(y_0)$$ 
where $\Phi(t)=t^{\frac{p-N}{p-1}}$, $r$ is the radius and $y_0$ the centre of $B$. 

Thus $\alpha$ is finite too, according to Lemma~\ref{lem:dndooo}, 
that is
\begin{equation}\label{4-lin-visc}
\alpha<\infty, \quad\beta<\infty.
\end{equation}

Recall that, using the notation in~\cite{Luis, LN1, LN2}, the free boundary condition 
takes the form \eqref{FB-cond-G}
$$ G(\alpha,\beta):= \alpha^p-\beta^p-\Lambda_0.$$
Therefore it is enough to show that
\begin{equation}\label{prove G}
\alpha^p-\beta^p=\Lambda_0.
\end{equation}
For this, we first consider the case~$\beta=0$, i.e. when~$u^-$ is degenerate.
We define the scaled function at~$x_0$
$$ u_{\rho}(x):=\frac{u(x_0+\rho x)}{\rho}, \quad 0<\rho<\dist(x_0, \partial \Omega).$$

Since $x_0$ is a non-flat point of free boundary 
then it follows from \eqref{4-lin-visc} that for any sequence~$\rho_j\to0$ as~$j\to+\infty$ there 
is a subsequence~$\rho_{j(k)}\to0$ such that~$u_{\rho_{j(k)}}$ converges 
to some~$u_0$. Moreover, owing to Lemma~\ref{lemma:linear}, 
in a non-tangential domain we have that
$$ u_{\rho}(x)=\alpha\langle x,\nu\rangle^++\frac{o(\rho|x|)}{\rho}\to
\alpha\langle x,\nu\rangle^+ \quad {\mbox{ as }}\rho\to 0.$$
Without loss of generality, we may assume that~$\nu=e_1$.
Thus, after blowing-up, we have that~$u_0=\alpha x_1^+$ in a 
cone~$K_0:=\{x=(x_1,x')\in \R\times\R^{N-1}, s.t. \  x_1\ge |x'|\cos\theta\}$ 
for some~$\theta\in (\frac\pi{4}, \frac\pi2)$.
Notice that~$u_0>0$ in~$K_0\cap (B_2\setminus B_1)$.
Also, $\Delta_p u_0=0$ in~$K_0\cap (B_2\setminus B_1)$.
Then, by the Unique Continuation Theorem (see Proposition~5.1 in~\cite{GM})
we get that~$u_0=\alpha x_1^+$ in~$\R^N$.
In turn, this implies that the free boundary condition is satisfied
in the classical sense on the hyperplane~$\{x_1=0\}$.
That is,~$|\nabla u_0|^p=\Lambda_0$ on~$\{x_1=0\}$,
and so~$\alpha^p=\Lambda_0$ on~$\{x_1=0\}$.
Hence~\eqref{prove G} is satisfied in the case~$\beta=0$.

Suppose now that~$\beta>0$, namely~$u^-$ is non-degenerate.
Reasoning as above and blowing-up, we can prove that~$u_0^+=\alpha x_1^+$.
It remains to show that~$u_0^-=\beta x_1^-$.
To do this, we set~$\Gamma_0:=\partial\{u_0>0\}$, that is~$\Gamma_0$
is the free boundary of the blow-up~$u_0$.
We take~$z\in\{x_1=0\}$, $z\ne0$, and we take the ball~$B_r(z)$ for some~$0<r<|z|$, see Figure~\ref{fig3}.

\begin{figure}
\begin{center}
\includegraphics[width=0.7\textwidth]{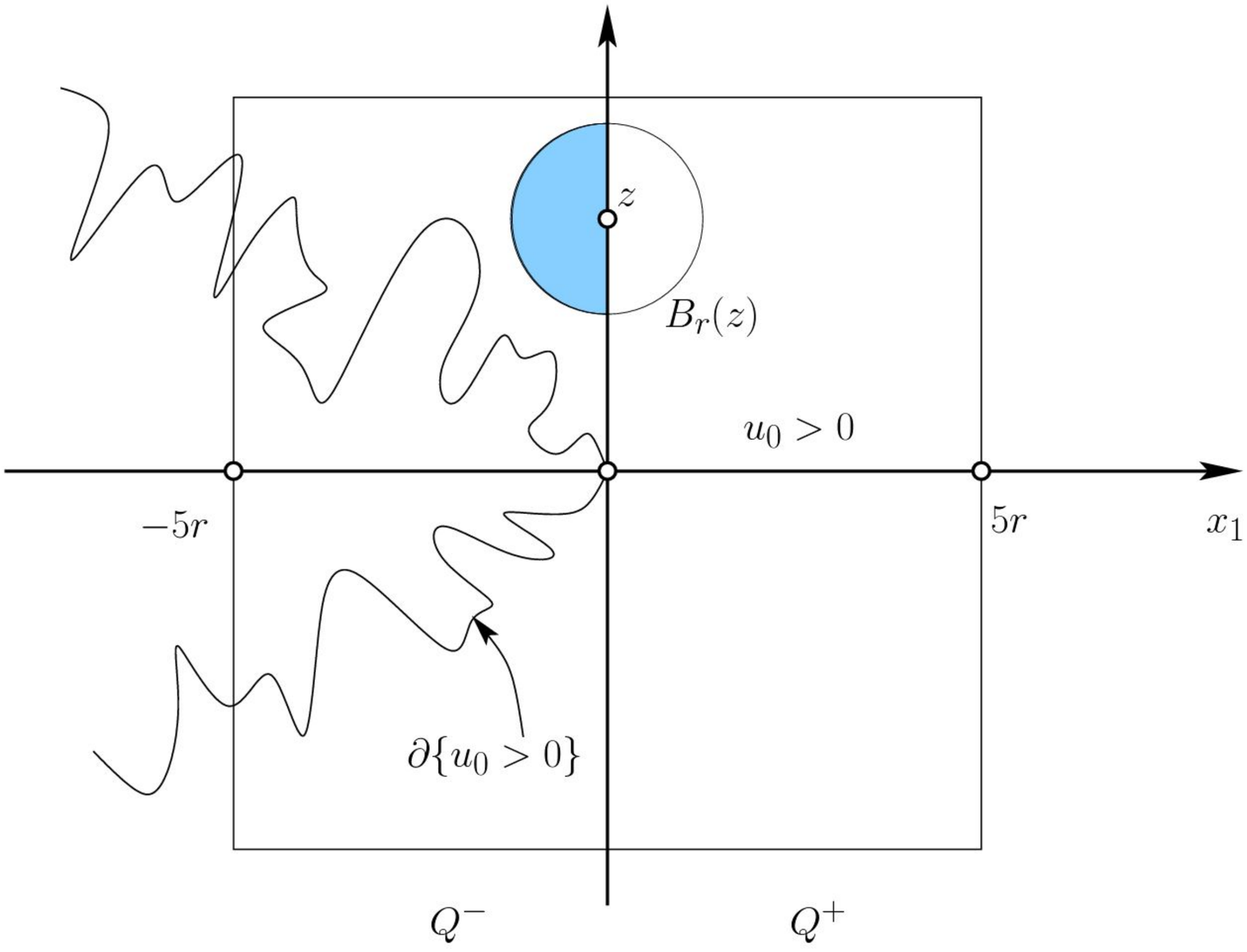}
\end{center}
\caption{In the coloured half of $B_r(z)$ either $u_0<0$ or $u_0\equiv0 $.}
\label{fig3}
\end{figure}
\medskip

There are  three possibilities: 
\begin{itemize}
\item[{\bf Case 1)}] $u_0$ vanishes only on~$B_r(z)\cap\{x_1=0\}$ and~$u_0>0$
in~$B_r(z)\cap \{x_1<0\}$,
\item[{\bf Case 2)}] $u_0$ vanishes only on~$B_r(z)\cap\{x_1=0\}$ and~$u_0<0$
in~$B_r(z)\cap \{x_1<0\}$,
\item[{\bf Case 3)}] $u_0$ vanishes in~$B_r(z)\cap\{x_1\le 0\}$.
\end{itemize}

Notice that {\bf Case 1)} cannot occur, because it would imply that  
we deal with a one-phase problem in~$B_r(z)$ and the 
$\mathcal H^N-$density estimate for the zero set would be  violated
(see Theorem~4.4 in~\cite{DP2}).

Consider now {\bf Case 2)}, and observe that on the hyperplane~$\{x_1=0\}$
the free boundary condition is satisfied in classical sense:
$$ |\nabla u_0^-|^p=\alpha^p-\Lambda_0=:\gamma^p$$
in $B_r(z)$. In particular,
\begin{equation}\label{boundary}
|\nabla u_0^-|=\gamma>0 \quad {\mbox{ on }}\{x_1=0\}\cap B_r(z).
\end{equation}
Define
$$
\tilde{u}_0:=\left\{
\begin{array}{ll}
u^-_0 & {\mbox{ in }}B_r(z)\cap\{x_1<0\},\\
\gamma x_1  & {\mbox{ in }}B_r(z)\cap\{x_1>0\}.
\end{array}
\right.
$$
We claim that
\begin{equation}\label{utilde p}
\Delta_p \tilde{u}=0 \quad {\mbox{ in }}B_{r/2}(z).
\end{equation}
Indeed, $\tilde{u}$ is~$p$-harmonic in~$B_r(z)\cap\{x_1>0\}$.
Moreover, \eqref{boundary} yields that $u_0^-\in C^2(\overline{B_{r/2}(z)\cap\{x_1<0\}})$, therefore
we have that~$\Delta_p\tilde{u}=0$ pointwise in~$\overline{B_{r/2}(z)\cap\{x_1<0\}}$,
and so~\eqref{utilde p} follows.

Hence, from \eqref{boundary}, \eqref{utilde p} and the Unique Continuation Theorem~\cite{GM}
we obtain that $u_0^-$ must be a linear function in $B_{r/2}(z)\cap\{x_1<0\}$.
Then, Proposition 5.1 in \cite{GM} implies that $u^-_0$ is a linear function
in $\{x_1<0\}$. Thus the free boundary condition is satisfied in the classical sense on the plane~$\{x_1=0\}$ including the origin,
and this proves equality in \eqref{prove G} in Case 2).

Now we deal with {\bf Case 3)}. We consider a cube~$Q=(-5r,5r)\times (-5r,5r)$
centered at the origin
such that~$B_r(z)\subset Q$, and we set~$Q^-:=Q\cap\{x_1<0\}$.
Notice that
\begin{equation}\label{aasddkjwe}
u_0\le 0 \quad {\mbox{ in }} Q^-.
\end{equation}
In particular, $u_0\le0$ on~$\partial Q^-$.
According to the remark in Subsection~\ref{rem:one phase}, $u_0$
is a minimizer in~$Q^-$ of the functional
$$ \tilde{J}(u)=\int_{Q^-}
|\nabla u|^p+ \Lambda \chi_{\{u>0\}}=
\int_{Q^-}|\nabla u|^p.$$
Therefore $u_0$ is~$p$-harmonic in~$Q^-$. By maximum principle,
$u_0$ cannot achieve its maximum inside~$Q^-$. This and~\eqref{aasddkjwe}
imply that~$u_0<0$ in~$Q^-$, and so the free boundary coincides with~$\{x_1=0\}$.

This concludes the proof of ii)-a) in Definition \ref{def:visc}.
Similarly, one can also prove  ii)-b).
Hence, $u$ is a viscosity solution, and  the desired result follows.
\end{proof}

%
%
\section{On $\e-$monotonicity of $u$ and slab flatness of $\fb{u}$}\label{sec:eps}

One of the main free boundary regularity theorems for viscosity solutions
is formulated in terms of the $\e-$monotonicity of $u$. More precisely, we have: 

\begin{defn}
We say that
$u$ is $\epsilon-$monotone
if there are a unit vector $e$ and an angle $\theta_0$ with 
$\theta_0 > \frac\pi 4$ (say) and $\epsilon >0$ (small)
such that, for every $\epsilon'\ge \epsilon $,
\begin{equation}\label{e-mon}
\sup_{B_{\epsilon' \sin\theta_0} (x)} u(y -\epsilon ' e) \le u(x).
\end{equation}
\end{defn}

We denote by~$\Gamma(\theta_0,e)$ the cone with axis~$e$ and opening~$\theta_0$. 

\begin{defn}
We say that $u$ is $\epsilon-$monotone in the 
cone~$\Gamma(\theta_0,\e)$ if it is~$\epsilon-$monotone 
in any direction~$\tau\in\Gamma(\theta_0,\epsilon)$. 
\end{defn}

One can interpret the $\e-$monotonicity of $u$ as closeness of 
the free boundary to a Lipschitz graph with Lipschitz constant sufficiently close to 
$1$ if we leave the free boundary in directions $e$ at distance  $\e$ and higher. 
The exact value of the Lipschitz constant is given 
by~$\left(\tan\frac{\theta_0}{2}\right)^{-1}$. 
Then the ellipticity propagates to 
the free boundary via Harnack's inequality giving that $\Gamma$ is Lipschitz. 
Furthermore, Lipschitz free boundaries are, in fact, $C^{1, \alpha}$ regular.

\medskip 

For $p=2$ this theory was founded by L. Caffarelli, 
see~\cite{Caffa1, Caffa3, Caffa2}. 
Recently J. Lewis and K. Nystr\"om proved that this  theory is valid for all $p>1$, see \cite{LN1, LN2}. In fact, 
their argument does not require $u$ to be Lipschitz. 

\medskip

For viscosity solutions we replace the $\epsilon-$monotonicity 
with the slab flatness measuring the thickness of $\fb u\cap B_r(x)$ in terms of the quantity 
$h(x,r)$ introduced in \eqref{min-h}. In other words, $h(x, r)$ measures
how close the free boundary is to a pair of parallel planes in a ball 
$B_r(x)$ with $x\in \Gamma.$ Clearly, planes are 
Lipschitz graphs in the direction of the normal, therefore 
the slab flatness of $\Gamma$ is a particular case of $\e-$monotonicity of $u$.
  
Hence, under $h_0-$flatness of the free boundary we can 
reformulate the regularity theory ``flatness implies $C^{1, \alpha}$'' as follows:
\begin{theorem}\label{TH1}
Suppose that~$B_r(x_0)\subset\Omega$ with $x_0\in \fb{u}$. Then there exists $h>0$ such that if
$\Gamma\cap B_r(x_0)\subset \{x\in \R^N : -hr<(x-x_0)\cdot\nu<hr\}$ 
then $\Gamma\cap B_{r/2}(x_0)$ is locally~$C^{1,\alpha}$ 
in the direction of $\nu$, for some~$\alpha\in(0,1)$.
\end{theorem}

%
%
\section{Linear growth vs flatness: Proofs of Theorems A and $\text A'$}\label{lin vs flat}

\subsection{Dyadic scaling} We first discuss a preliminary result, 
that we will use for the proof of Theorem A. 

\begin{prop}\label{prop:lin flat}
Let~$u$ be a local minimizer of~$J$ 
and~$x_0\in\Gamma\cap B_1\subset\Omega$. For any~$k\in\N$, set
$$
S(k,u):=\sup_{B_{2^{-k}}^+(x_0)}|u|.
$$

If $h_0>0$ is fixed and $h\left(x_0,\frac1{2^k}\right)\ge \frac{h_0}{2^{k+1}}$ for some~$k$, then
\begin{equation}\label{discrete-linear}
S(k+1,u)\leq\max\left\{\frac{L2^{-k}}{2},\frac{S(k,u)}{2},\dots, 
\frac{S(k-m,u)}{2^{m+1}},\dots, \frac{S(0,u)}{2^{k+1}}\right\},
\end{equation}
for some positive constant $L$, that is independent of~$x_0$ and $k$. 

Otherwise if $h\left(x_0, \frac1{2^k}\right)< \frac{h_0}{2^{k+1}}$ 
for some~$k$, then
$\Gamma\cap B_{2^{-(k+1)}}$ is a $C^{1,\alpha}$ smooth surface, for some $\alpha\in (0,1)$.
\end{prop}

\begin{proof} 
We first deal with 
the case~$h\left(x_0,\frac1{2^k}\right)\ge \frac{h_0}{2^{k+1}}$. 
In order to prove~\eqref{discrete-linear}, 
we use a contradiction argument discussed in~\cite{K1}. 
Hence, we suppose that~\eqref{discrete-linear} fails, that is 
there exist integers~$k_j, j=1,2,\ldots$, local minimizers~$u_j$ 
and points~$x_j\in\Gamma_j\cap B_1$ such that 
\begin{equation}\label{level}
h\left(x_j,\frac1{2^{k_j}}\right)\ge \frac{h_0}{2^{k_j+1}}
\end{equation}
and
\begin{equation}\label{neg}
S(k_j+1,u_j)>\max\left\{\frac{j2^{-k_j}}{2},\frac{S(k_j,u_j)}{2},\dots, 
\frac{S(k_j-m,u_j)}{2^{m+1}},\dots,\frac{S(0,u_j)}{2^{k_j+1}}\right\}.
\end{equation}

Since~$u_j$ is a local minimizer of~$J$ in~$B_1$ and~$u_j(x_j)=0$, 
then~$u_j$ is bounded (see Theorem~1 in~\cite{K1}). Namely, there exists 
a positive constant~$M$, that is independent of~$j$, 
such that~$S(k_j+1,u_j)\le M$. Therefore, from~\eqref{neg} 
we have that~$M\ge j2^{-k_j}/2$,
which implies that~$2^{k_j}\ge j/(2M)$. 
Hence, $k_j$ tends to~$+\infty$ when~$j\to+\infty$. 

We set
\begin{equation}\label{sigma}
\sigma_j:=\frac{2^{-k_j}}{S(k_j+1,u_j)}. 
\end{equation}
Using~\eqref{neg} once more, we see that 
\begin{equation}\label{cc18}
\sigma_j <\frac{2}{j}\to 0 \ {\mbox{ as }}j\to+\infty.
\end{equation}

For any~$j$, we now define the function
\begin{equation}\label{scaled v}
v_j(x):=\frac{u_j(x_j+2^{-k_j}x)}{S(k_j+1,u_j)}.
\end{equation}
Then, by construction, 
\begin{equation}\label{cc1}
\sup_{B_{1/2}}|v_j|=1.
\end{equation}
Furthermore, from~\eqref{neg} we have that 
$$ 1>\max\left\{\frac{j2^{-k_j}}{2S(k_j+1,u_j)},\frac{1}{2}\sup_{B_1}|v_j|,\dots, 
\frac{1}{2^{m+1}}\sup_{B_{2^m}}|v_j|,\dots,\frac{1}{2^{k_j+1}}\sup_{B_{2^{k_j+1}}}|v_j|\right\},$$
which in turn implies that 
\begin{equation}\label{cc2}
\sup_{B_{2^m}}|v_j|\le 2^{m+1}, \quad \text{for any}\ m<2^{k_j}.
\end{equation}
Finally, since~$u_j(x_j)=0$, we have that
\begin{equation}\label{cc3} v_j(0)=0.\end{equation} 

Notice that~$v_j$ is a minimizer (according to its own boundary values) 
of the scaled functional
\begin{equation}\label{scaled J}
\widehat{J}(v):= \int_{B_{R}}|\nabla v|^p +\sigma_j^p\lambda(v),
\end{equation}
for $0<R<2^{k_j}$ and~$j$ large. Indeed, from~\eqref{scaled v} and 
an easy computation, we get 
$$ \nabla v_j(x)=\frac{2^{-k_j}}{S(k_j+1,u_j)}\nabla u_j(x_j+2^{-k_j}x).$$
Hence, by the change of variable~$y=x_j+2^{-k_j}x$ and recalling~\eqref{sigma},
\begin{eqnarray*}
\widehat{J}(v_j)&=& \int_{B_{R}}|\nabla v_j(x)|^p +\sigma_j^p\lambda(v)\,dx\\
&=& \int_{B_{R}}\frac{2^{-pk_j}}{S(k_j+1,u_j)^p}|\nabla u_j(x_j+2^{-k_j}x)|^p +\sigma_j^p\lambda(u_j(x_j+2^{-k_j}x))\,dx\\
&=& \sigma_j^p\,2^{nk_j} \int_{B_{R2^{-k_j}}(x_j)}|\nabla u_j(y)|^p+\lambda(u_j)\,dy.
\end{eqnarray*}
Since $u_j$ is a minimizer for~$J$, the last formula implies 
that~$v_j$ is a minimizer for~$\widehat{J}$. Hence, from Lemma~\ref{lemma:coherence} we obtain that
for any~$q>1$ and~$0<R<2^{k_j}$ there exists a 
constant~$C=C(R,q)>0$ independent of~$j$ such that 
$$
\max\{\|v_j\|_{C^{\alpha}(B_R)},\|\nabla v_j\|_{L^q(B_R)}\}\le C,
$$
for some~$\alpha\in(0,1)$. 
Therefore, by a standard compactness argument, we have that, up to a subsequence,
\begin{equation}\label{choice q}
{\mbox{$v_j$ 
converges to some function~$v$ as~$j\to+\infty$ 
in~$W^{1,q}(B_R)\cap C^{\alpha}(B_R)$ for any fixed~$R$.}}\end{equation} 
 
From~\eqref{cc1}, \eqref{cc2} and~\eqref{cc3} we obtain that 
$$
\sup_{B_{1/2}}|v|=1,\quad \sup_{B_{2^m}}|v|\le 2^{m+1} \quad 
{\mbox{and }} \quad v(0)=0.$$
We claim that 
\begin{equation}\label{claim v}
{\mbox{$v$ is a minimizer for the functional $\mathcal{J}(v):=\int_{B_R}|\nabla v|^p$.}}
\end{equation}
For this, notice that for any~$\varphi\in C^\infty_0(B_R)$
\begin{equation}\label{cc20}
\int_{B_R}|\nabla v_j|^p+\sigma_j^p\lambda(v_j)\le 
\int_{B_R}|\nabla (v_j+\varphi)|^p+\sigma_j^p\lambda(v_j+\varphi), 
\end{equation}
because ~$v_j$ is a minimizer for~$\widehat{J}$ defined in~\eqref{scaled J}.
By taking~$q>p$ in~\eqref{choice q}, we have that 
\begin{eqnarray*}
&&  \int_{B_R}|\nabla v_j|^p\to \int_{B_R}|\nabla v|^p \\
{\mbox{and }} &&  \int_{B_R}|\nabla (v_j+\varphi)|^p\to 
\int_{B_R}|\nabla (v+\varphi)|^p
\end{eqnarray*}
as $j\to+\infty$. Moreover, from~\eqref{cc18} we obtain 
\begin{eqnarray*}
\int_{B_R}\sigma_j^p\lambda(v_j)\to 0\quad 
{\mbox{and }} \quad \int_{B_R}\sigma_j^p\lambda(v_j+\varphi)\to 0
\end{eqnarray*}
as~$j\to+\infty$. Thus, sending~$j\to+\infty$ in~\eqref{cc20} 
and using these observations, we get 
$$ \int_{B_R}|\nabla v|^p\le \int_{B_R}|\nabla (v+\varphi)|^p$$
for any~$\varphi\in C^\infty_0(B_R)$. 
This implies~\eqref{claim v}.

Hence, from Liouville's Theorem (see Theorem~\ref{th:liouville}) 
we deduce that~$v$ must be a linear function in $\R^N$.
Without loss of generality we can take $v(x)=Cx_1$ for some positive constant
$C$. 

On the other hand, \eqref{level} implies that the following inequality holds true 
for the function~$v_j$: 
$$ h(0,1)\ge \frac{h_0}{2}.$$
By the uniform convergence in \eqref{choice q}, 
we have that for any $\epsilon >0$ there is $j_0$ such that $|Cx_1-v_j(x)|<\epsilon$
whenever $j>j_0$.  Since $\fb {v_j}$ is $h_0/2$ thick in $B_1$ it follows that 
there is $y_j\in \fb {v_j}\cap B_1$ such that $y_j=e_1h_0/4+t_je'$, for some $t_j\in \R$, where 
$e_1$ is the unit direction of $x_1$ axis and $e'\perp e_1$. Then we have that 
$|C\frac{h_0}4-0|=|v(y_j)-v_j(y_j)|<\epsilon$, which is a contradiction if $\epsilon$ is small.
This finishes the proof of \eqref{discrete-linear}.

\medskip

If $h(x_0,\frac1{2^k})< \frac{h_0}{2^{k+1}}$ for some~$k$, we use 
Theorem~\ref{TH:viscosity} to obtain that~$u$ is also a viscosity solution 
in the sense of Definition~\ref{def:visc}. Therefore, we can apply the regularity 
result in Theorem~\ref{TH1}, thus obtaining the desired claim. 
This concludes the proof of Proposition~\ref{prop:lin flat}.
\end{proof}

\subsection{Proof of Theorem A} 

With the aid of Proposition~\ref{prop:lin flat} we now complete the proof of Theorem A.

\begin{proof}[Proof of Theorem A]
The argument in Proposition~\ref{prop:lin flat} shows that 
either there are finitely many integers~$k$ such that
\begin{equation}\label{ff1}
h\left(x_0,\frac1{2^k}\right)\ge \frac{h_0}{2^{k+1}}
\end{equation} 
and 
\begin{equation}\label{ff2}
S(k+1,u)\leq\max\left\{\frac{L2^{-k}}{2},\frac{S(k,u)}{2},\dots, 
\frac{S(k-m,u)}{2^{m+1}},\dots, \frac{S(0,u)}{2^{k+1}}\right\},
\end{equation}
or there are infinitely many~$k$ such that~\eqref{ff1} and~\eqref{ff2} 
hold true. 

In the first case, there exists~$k_0$ such 
that~$h\left(x_0,\frac1{2^{k_0}}\right)< \frac{h_0}{2^{k_0+1}}$, 
and so~$\Gamma\cap B_{2^{-(k_0+1)}}$ is a $C^{1,\alpha}$ smooth surface. 
In the second case, we have linear growth of~$u$ at 
the free boundary point~$x_0$ where the flatness does not improve. 

Suppose now that we are given~$r>0$. Then, either~$h(x_0,r)<\frac{h_0}2 r$ 
or~$h(x_0,r)\ge \frac{h_0}2  r$. In the first case, we obtain that~$\Gamma$ is 
a~$C^{1,\alpha}$-surface. In the second case we argue as follows: 
there exists~$k\in\N$ such that
$$ \frac{1}{2^{k+1}}\le r\le \frac{1}{2^k}.$$ 
Hence, by the definition of~$h$ given in~\eqref{min-h}, we have that 
$$ h\left(x_0,\frac{1}{2^{k+1}}\right)\ge h(x_0,r)
\ge \frac{h_0}2 r\ge \frac{h_0}2 \frac{1}{2^{k+1}}.$$  
This means that we are in the position to apply Proposition~\ref{prop:lin flat}, 
that implies linear growth of~$u$ at the level $r/2$.
\end{proof}

A refinement of Theorem~A is given by the following:

\begin{corollary}\label{cor:lin flat}
Let~$h_0$ be the constant given in Theorem A. 
Then, if~$r\in[2^{-k-1}, 2^{-k})$ and $h(x_0,r)\ge \frac{h_0}2 r$, we have that 
$$ \sup_{B_{\frac{r}2}(x_0)}|u|\leq 2Lr,$$
where $L$ is the constant given by Theorem A.
\end{corollary}

%
%
\subsection{Alt-Caffarelli-Friedman functional}\label{sec:ACF-func} 
Here we introduce a functional that is a generalization to any $p>1$ 
of the one introduced by Alt, Caffarelli and Friedman in the case $p=2$, 
and we show that this functional is bounded at non-flat free boundary points, 
thanks to the linear growth ensured by Theorem A. 

For this, we let $u=u^+-u^-$, where $u^+:=\max\{0, u\}$ and~$u^-:=-\min\{0, u\}$.
We define the functional  
\begin{equation*}
\varphi_p(r,u,x_0):=\frac{1}{r^4}\int_{B_r(x_0)}\frac{|\nabla u^+|^p}{|x-x_0|^{N-2}}\,\int_{B_r(x_0)}\frac{|\nabla u^-|^p}{|x-x_0|^{N-2}}
\end{equation*}
where $x_0\in \fb{u}$ and $r>0$ is such that $B_r(x_0)\subset \Omega$. 

Precisely, we show the following: 

\begin{corollary}
Let $h_0$ be fixed, $D\Subset \Omega$ be a subdomain and $x_0\in \Gamma \cap D$ 
be such that $h(r, x_0)\ge \frac{h_0}{2}r$. 

Then there exist $M>0, r_0>0$ depending only on $N, p, h_0, \sup_\Omega|u|, \Lambda$ and $\dist(D, \partial \Omega)$ 
such that  
\begin{equation*}
\phi_{p}(r, u, x_0)\leq \frac{M^2N^2}4, \quad \forall r\leq r_0.
\end{equation*}
\end{corollary}

\begin{proof}
Since $u^\pm$ is nonnegative $p-$subsolution 
(recall P.1 in Proposition~\ref{prop:tec}), 
we can apply Caccioppoli's inequality, obtaining that
$$ \int_{B_{\rho}(x_0)}|\nabla u^\pm|^p\leq \frac{C}{\rho^p}\int_{B_{2\rho}(x_0)}(u^\pm)^p.$$ 
From this and Corollary~\ref{cor:lin flat} we have that 
$$\int_{B_\rho(x_0)}|\nabla u^\pm|^p\leq M\rho^N,$$
for some $M>0$. Hence, using  Fubini's Theorem, we have
\begin{eqnarray*}
\int_{B_r(x_0)}\frac{|\nabla u^\pm|^p}{|x-x_0|^{N-2}}&=&\int_0^r\frac1{\rho^{N-2}}\int_{\partial B_\rho(x_0)}|\nabla u^\pm|^p \\
&=&\frac1{r^{N-2}}\int_{B_r(x_0)}|\nabla u^\pm|^p+(N-2)\int_0^r\frac1{\rho^{N-1}}\int_{B_\rho(x_0)}|\nabla u^\pm|^p\\
&\le&\frac{MN}2 r^2,
\end{eqnarray*}
which implies the desired result. 
\end{proof}

\begin{remark}
In \cite{DK2} we prove the converse statement in some sense. More precisely 
we show that if $N=2$ and $p>2$ is close to $2$ then $\phi_p(r, u, x_0)$ is discrete monotone.
\end{remark}

%
%
\section{Partial Regularity: Proof of Theorem B}\label{sec:thb}

In this section we introduce the set-up in order to prove Theorem B. 
For this, we recall the notation introduced in Section \ref{sec:main} 
(recall in particular Definition \ref{def-N-set} and formula \eqref{lyukiy0909090}). 
We first show that $\Delta_p u^+$ is Radon measure.

\begin{lemma}\label{lem-7.11}
Let $u$ be a local minimizer of~\eqref{Ju}. 
Then, the following statements hold true. 
\begin{itemize}
\item $\Delta_p u^+$ is a Radon measure and, 
for any~$x\in \Gamma:=\fb u$ and~$r>0$ such that~$B_{2r}(x)\subset \Omega$, there holds 
\begin{equation}\label{p-lap-nondeg-BIS}
\int_{B_r(x)}\Delta_p u^+\leq \frac1r\int_{B_{2r}(x)}|\nabla u^+|^{p-1}.
\end{equation} 
\item For a given subdomain $D\Subset \Omega$  there is $r_0>0$ such that 
\begin{equation}\label{p-lap-nondeg}
\int_{B_r(x)}\Delta_p u^+ \ge Cr^{N-1}, \quad {\mbox{ for any }} r<r_0, 
\end{equation}
for all 
$x\in \mathscr N \cap D$, where $C>0$ depends on $\Lambda$, $N$, $p$, 
$\dist(D, \partial \Omega)$ and $L$ (given by Theorem~A).
\item  For each $x\in \mathscr F$ there is $r(x)>0$ such that 
\begin{equation}\label{p-lap-nondeg-0}
\int_{B_r(x)}\Delta_p u^+ \ge Cr^{N-1}, \quad {\mbox{ for any }} r<r(x), {\mbox{ with }} 
B_{r(x)}(x)\subset \Omega,
\end{equation}
for some $C>0$ that depends on $\Lambda$, $N$, $p$, 
$\dist(D, \partial \Omega)$ and $x$.
\end{itemize}
\end{lemma}

\begin{proof}
We first show \eqref{p-lap-nondeg-BIS}. For this, we take for simplicity $x=0$. 
Observe that by P.1 in Proposition \ref{prop:tec} we have that $\Delta_pu^+\ge 0$ 
in the sense of distributions. Also, for any $\rho\in(r,2r)$, 
$$\int_{B_\rho}\Delta_p u^+=\int_{ \partial B_\rho}|\nabla u^+|^{p-2}\partial_\nu u^+.
$$
Therefore, integrating both sides of the last identity over the interval $(r, 2r)$ with respect 
to $\rho$, we infer that
\begin{eqnarray*}
r\int_{B_r}\Delta_p u^+&\leq& \int_{r}^{2r} \int_{B_\rho}\Delta_p u^+
=\int_{r}^{2r}\int_{ \partial B_\rho}|\nabla u^+|^{p-2}\partial_\nu u^+\\\nonumber
&=&\int_{B_{2r\setminus B_r}} |\nabla u^+|^{p-2}\nabla  u^+\cdot \frac{x}{|x|}\\\nonumber
&\leq & \int_{B_{2r}}|\nabla u^+|^{p-1}.
\end{eqnarray*}
This proves~\eqref{p-lap-nondeg-BIS}. 
\smallskip

To prove \eqref{p-lap-nondeg} we argue towards a contradiction.
So, for any $j=1,2,\ldots$, we let $x_j\in \mathscr N$ and~$r_j>0$ such that 
\begin{equation}\label{argue}
\int_{B_{r_j}(x_j)}\Delta_p u^+ < \frac{r_j^{N-1}}j.\end{equation} 
We also introduce~$v_j(x):=\frac{u(x_j+r_j x)}{r_j}$. 

Since $x_j\in \mathscr N$, it follows from Theorem A that 
$u$ has  uniform linear growth at $x_j$. This property translates to 
the scalings of $v$ at $x_j$ giving 
uniform linear growth for the functions $v_j$ at the origin, i.e.
$|v_j(x)|\leq L|x|$ where $L$ is the constant in Theorem A.

Notice that~$v_j$ is a minimizer of~\eqref{Ju}, so it is locally~$C^\alpha$, for some~$\alpha\in(0,1)$, 
thanks to Lemma~\ref{lemma:coherence}. Hence~$\{v_j\}$ is uniformly bounded in $C^{1,\alpha}$, 
and so is~$\|\nabla v_j\|_{L^p(B_M)}$ for any fixed $M>0$, thanks to Caccioppoli's inequality.  
Therefore, we can extract a subsequence 
$\{r_{j(m)}\}$ such that $v_{j(m)}\to v_0$ as $m\to+\infty$ and~$v_0$ is a minimizer of $J$ in $B_2$. 
Moreover, by~\eqref{argue}, 
$$\int_{B_1}\Delta_p v_0^+=0,$$ 
with  $v_0^+(0)=0$. As a consequence, $v_0^+$ vanishes identically in $B_1$, 
by the minimum principle for the $p-$harmonic functions. 
On the other hand, from Corollary   
\ref{cor-nondeg} we have that~$\sup_{B_{\frac12}}v_0^+\ge \frac c2$, 
and this gives a contradiction. 
Thus the proof of~\eqref{p-lap-nondeg} is finished as well.
\smallskip 

The  proof of the non-uniform estimate \eqref{p-lap-nondeg-0} follows from 
a similar argument, by replacing $L$ with a constant 
$C(x)$ depending on $\na u^+(x)$ and $\na u^-(x)$. 
\end{proof}

As a consequence of Lemma \ref{lem-7.11}, we obtain the first part of Theorem B. 
More precisely:

\begin{corollary} \label{coro:zero}
Let~$R>0$ be such that $B_R\subset \Omega$. 
Then $\mathcal H^{N-1}(\fb{u} \cap B_R)<\infty$. 
\end{corollary}

\begin{proof}
It follows from \eqref{p-lap-nondeg} and \eqref{p-lap-nondeg-0} that 
for each $x\in \Gamma\cap B_R$ there is $r(x)>0$ such that 
\begin{equation}\label{blya}
\int_{B_r(x)}\Delta_p u^+ \ge Cr^{N-1}, \quad \text{whenever}\quad   r<r(x). 
\end{equation}
Thus $ \cup_{x\in\Gamma\cap B_R} B_{r(x)}(x)$ is a Besicovitch type covering of 
$\Gamma\cap B_R$.
Applying  Besicovitch's Covering Lemma, we have that there is a subcovering 
$\displaystyle \mathcal F=\bigcup_{k=1}^{m(N)} \mathcal G_k$ of balls
$B_i:=B_{r(x_i)}(x_i)$ such that $\sum_i \chi_{B_i}\leq A$
for some dimensional constant $A>0$ and
$$\Gamma\cap B_R \subset \bigcup_{k=1}^{m(N)}\bigcup_{B_i\in \mathcal G_k} B_i,$$
where the balls $B_i$ in each $\mathcal G_k$ are disjoint and $\mathcal G_k$ are countable.

Now we take a small number $\delta>0$, and we observe that if $r(x)>\delta$ then \eqref{blya}
holds for any $r<\delta$. Hence, without loss of generality, 
we take $r(x)<\delta$ for any $x\in\Gamma\cap B_R$. 

Therefore, using~\eqref{blya},  
\begin{eqnarray*}
 C\sum_{B_i\in \mathcal F}r_i^{N-1}&\leq &\sum_{B_i\in\mathcal F}\int_{B_i} \Delta_pu^+\\
&=&\sum_{k=1}^{m(N)}\sum_{B_i\in \mathcal G_k} \int_{B_i} \Delta_pu^+\\
&\leq& A\, m(N)\int\limits_{B_{8\delta}(\Gamma\cap B_R)}\Delta_pu^+,
\end{eqnarray*} 
where~$B_{8\delta}(\Gamma\cap B_R)$ is the $8\delta$ neighbourhood of
$\Gamma\cap B_R$. 
Thus, choosing a finite covering of~$B_{8\delta}(\Gamma\cap B_R)$ 
with balls $B_{R_0}(z_j)$, with $j=1, \dots, \ell$, such that 
$B_{2R_0}(z_i)\subset \Omega$ and  $B_{8\delta}(\Gamma\cap B_R)\subset \cup_{j=1}^\ell B_{R_0}(z_j)$ 
and using~\eqref{p-lap-nondeg-BIS}, we have that
$$ \mathcal  H^{N-1}_\delta(\Gamma\cap B_R)\leq \frac{A}{C2^{N-1}}\,\frac{1}{R_0}
\sum_{j=1}^\ell\int_{B_{2R_0}(z_j)}|\nabla u^+|^{p-1}<+\infty,$$
and letting $\delta\rightarrow 0$ we arrive at the desired result. \end{proof}

We end this section by the following density type estimate to be used 
in the final stage of the proof of Theorem B. 

\begin{lemma}\label{density est}
For any subdomain $D\Subset \Omega$ there is a positive constant $c\in (0,1)$
depending on $N, p, \Lambda, \sup_\Omega|u|$ and $\dist(D, \partial\Omega)$
such that 
\begin{equation}\label{neg-dens}
\liminf_{r\to 0} \frac{|\{u\le 0\}\cap B_r(x_0)|}{|B_r(x_0)|}\ge c, \quad 
{\mbox{ for any }} x_0\in D\cap \fb u.
\end{equation}
\end{lemma}

\begin{proof}
Notice that if $x_0\in \mathscr F\cap D$ then \eqref{neg-dens} holds true with $c=1/2$. 
So we focus on the case in which $x_0\in \mathscr N\cap D$. 

We fix $r>0$ such that $B_r(x_0)\subset \Omega$,
and we take a function $v_r$ that is $p-$harmonic in $B_r(x_0)$
and such that $u=v_r$ on $\partial B_r(x_0)$. 
Then, reasoning as at the beginning of the proof of Lemma \ref{lemma:coherence} 
(in particular, using \eqref{Donatel-1}, \eqref{Donatel-2}, \eqref{Donatel-3}
and \eqref{Duzaa-inq}), 
we have that there exists a tame constant $\bar c>0$ such that 
\begin{equation}\begin{split}\label{rezzza}
\bar c \int_{B_r(x_0)}|V(\na u)-V(\na v_r)|^2\,\leq& \int_{B_r(x_0)}|\na u|^p-|\na v|^p\\
\,\le&\int_{B_r(x_0)}\Lambda(\X{v_r>0}-\X{u>0})\\
\,\le& \int_{B_r(x_0)}\Lambda\X{u\leq 0}. 
\end{split}\end{equation}

Now we claim that there is a constant $\Theta>0$ independent of $r$ such that 
\begin{equation}\label{Thet-007}
v_r(x_0)\ge \Theta r\quad \text{and} \quad   \int_{B_r(x_0)}|V(\na u)-V(\na v_r)|^2\ge \frac{\Theta}{r^p}\int_{B_r(x_0)}|u-v_r|^p.
\end{equation}
Notice that by comparison principle 
it follows that $v_r(x_0)\ge u(x_0)=0$. 
We prove the first 
inequality in \eqref{Thet-007} using a contradiction argument based on compactness, the 
second one can be proved analogously. 

Suppose that, for any $j=1,2,\ldots$, there are $x_j\in D\cap \mathscr N$ and $r_j>0$ 
with $B_{2r_j}(x_j)\subset \Omega$ such that 
\begin{equation}\label{ewtoyu}
0<v_j(x_j)\leq \frac{r_j}j.\end{equation} 
Now, define $\tilde v_j(x):=\frac{v_j(x_j+r_jx)}{r_j}$ and 
$\tilde u_j(x)=\frac{u(x_j+r_j x)}{r_j}$, for any $x\in B_1$. 
We recall that \eqref{mek-qash} implies that $\tilde u_j$ is a minimizer for $J$ in $B_1$. 
So, it follows from P.1 in Proposition \ref{prop:tec}, Caccioppoli's inequality 
and Theorem A that 
\begin{equation}\label{vero}
\int_{B_1}|\na \tilde u_j^\pm|^p\le C(N)\int_{B_2}(\tilde u_j^\pm)^p\leq C(N)\omega_N2^{N+2p}L^p,
\end{equation}
where $L$ is the constant introduced in Theorem A. 

Also, we observe that $\Delta_p \tilde v_j=0$ in $B_1$ and that 
$\tilde v_j=\tilde u_j$ on $\partial B_1$. In particular, $\int_{B_1}|\na \tilde v_j|^p\le\int_{B_1}|\na \tilde u_j|^p.$ 
This and \eqref{vero} imply that 
$\|\tilde v_j\|_{W^{1, p}(B_1)}\leq C(N)L^p$, up to renaming $C(N)$ 
(recall that $\tilde u_j^\pm$ are $p$-subharmonic, thanks to P.1 in Proposition \ref{prop:tec}).

Moreover, from the local regularity theory for 
$p-$harmonic functions we have that $\tilde v_j$ are uniformly $C^{1, \alpha}$ in $B_{\frac12}$.
Consequently, we have that there is a subsequence (still denoted by $\tilde v_j$)
such that $\tilde v_j\to v_0$ weakly in $W^{1, p}(B_1)$ and uniformly in $B_{\frac12}$, 
as $j\to+\infty$. 
In particular, by \eqref{ewtoyu}, 
$$ v_0(0)=\lim_{j\to \infty}\tilde v_j(0)= 0.$$  

As for the sequence $\tilde u_j$, from \eqref{mek-qash} and Lemma \ref{lemma:coherence}
we infer that there is a subsequence (still denoted by 
$\tilde u_j$) such that $\na \tilde u_j\to \na u_0$ strongly in $L^q(B_1)$ for any $q>1$
and $\tilde u_j\to u_0$ uniformly in $\overline {B_1}$, as $j\to+\infty$.
Furthermore, $u_0$ is a minimizer of $J$ and from the convergence of traces 
it follows that $v_0=u_0$ on $\partial B_1$. 
Also, by Corollary \ref{cor-nondeg} we have that $u_0\not =0$, 
and by Proposition \ref{prop:tec} we have that $u_0$ is $p-$subharmonic in $B_1$. 

Altogether, we have obtained that 
$$ \Delta_p v_0\le\Delta_p u_0 \; {\mbox{ in }} B_1, \quad v_0=u_0 \; {\mbox{ on }}\partial B_1 
\quad {\mbox{ and }} v_0(0)=u_0(0)=0.$$ 
But this is a contradiction to the comparison principle 
for $p$-harmonic functions. 

The second inequality of \eqref{Thet-007} can be proven analogously.

Now we are ready to finish the proof of \eqref{neg-dens}. From \eqref{rezzza} and \eqref{Thet-007}
we have 
\begin{eqnarray}\label{7.15000}
\int_{B_r(x_0)}\Lambda\X{u\leq 0}&\ge&  \frac{\bar c \Theta}{r^p}\int_{B_r(x_0)}|u-v_r|^p\\\nonumber
&\ge&\frac{\bar c \Theta}{r^p}\int_{B_{\kappa r}(x_0)}|u-v_r|^p
\end{eqnarray}
for $0<\kappa<1$ to be chosen later. 
Observe that by standard gradient estimates 
$$|\na v_r(y)|\le \frac{C}{1-\kappa}\frac{\sup_{B_r(x_0)}|v_r|}{r}\leq \frac{CL}{1-\kappa}, 
\quad y\in B_{\kappa r}(x_0),$$
up to renaming $C>0$, 
where the last inequality follows from the maximum principle and 
Theorem A. Therefore, for any $y\in B_{\kappa r}(x_0)$
\begin{eqnarray}
|v_r(y)-u(y)|&\ge& |v_r(0)-u(y)|-|v_r(y)-v_r(0)|\\\nonumber
&\ge&v_r(0)-|u(y)|-|v_r(y)-v_r(0)|\\\nonumber
&\ge& v_r(0)-2L\kappa r-\frac{CL}{1-\kappa}\kappa r\\\nonumber
&\ge&r\left(\Theta-\kappa L\left(2+\frac{C}{1-\kappa}\right)\right)\\\nonumber
&\ge& r\frac{\Theta}{2}
\end{eqnarray}
if we choose $\kappa$ small enough. Returning to 
\eqref{7.15000} we finally get that 
$$\frac{|\{u\le 0\}\cap B_r(x_0)|}{|B_r(x_0)|}\ge \frac{\bar c\Theta^{p+1}}{2^p\Lambda}\kappa^N.$$ 
This finishes the proof of Lemma \ref{density est}. 
\end{proof}

%
%
\section{Blow-up sequence of $u$, end of proof of Theorem B}\label{sec:blow}

In this section we study the blow-up sequences 
of a minimizer of~\eqref{Ju} and prove a simple compactness result, 
that we use to conclude the proof of Theorem B. 
For this, let $u$ be a minimizer of $J$ and $x_0\in \fb{u}$. 
Consider a sequence of balls
$B_{\rho_k}(x_0)$, with~$\rho_k\to 0$. 
We call the sequence of functions defined by 
\begin{equation}
u_k(x)=\frac{u(x_0+\rho_kx)}{\rho_k}
\end{equation}
the blow-up sequence of $u$ with respect to $B_{\rho_k}(x_0)$. 
Clearly $u_k$ is also a local minimizer.

\begin{prop}\label{tech-2}
Let $x_0\in \mathscr N$ and $u_k$ be a blow-up sequence. Then there is 
a blow-up limit $u_0:\R^N\to \R$ with linear growth such that 
for a subsequence
\begin{itemize}
\item $u_k\to u_0$ in $C_{loc}^\alpha(\R^N)$ for any $\alpha\in (0, 1)$, 
\item $\na u_k\to \na u_0$ weakly in $W^{1, q}$ for any $q>1$,
\item  $\fb{u_k}\to \fb{u_0}$ locally in Hausdorff distance, 
\item $\X{u_k>0}\to \X{u_0>0}$ in $L_{loc}^1(\R^N)$.
\end{itemize}
\end{prop}

\begin{proof}
The first and second claims follow from Lemma \ref{lemma:coherence} 
and a customary compactness argument to show that the blow-up limit $u_0$  exists. 

We recall the definition of Hausdorff distance: 
$$d_{\mathcal H}(F,G):=
\inf\left\{\delta: F\subset \bigcup_{x\in G} B_\delta(x), G\subset \bigcup_{x\in F} B_\delta(x) \right\}.$$  
Let $B_r:=B_r(z_0)$ be a ball not intersecting $\partial \{ u_0>0\}$. 
If $u_0>0$ in $\overline{B_r}$ then, by locally uniform convergence,  $u_k>0$ in $B_{\frac r2}$, 
thus implying that $\partial\{u_k>0\}\cap B_{r/2}=\emptyset$. As for the case  $u_0\leq 0$ in $B_r$, 
it follows from Proposition \ref{str-nd-hi} that $\frac1r\fint_{B_r}(u^+_0)<\e$, for any small 
$\e>0$. Thus, by the uniform convergence, we have that 
$\frac1r\fint_{B_r}(u^+_k)<\e$ if $k$ is sufficiently large.
From  Proposition \ref{str-nd-hi} we conclude that $u_k\leq 0$ in $B_{r/2}$. In both cases we infer that 
$\fb{u_k}$ does not intersect $B_{r/2}$ if $k$ is large enough.

\smallskip

Conversely, if $B_r$ does not intersect $\fb{u_k}$ for any large $k$, then 
either $u_k>0$ in $B_r	$ or $u_k\le 0$ in $B_r$. In the first case,
$u_k$ is $p-$harmonic in $B_r$ and hence so is $u_0$. Consequently, either $u_0>0$
in $B_r$ or $u_0\equiv 0$ in $B_r$. Thus $B_r$ does not intersect $\fb{u_0}$. 
In the second case, we have that $u_0\leq 0$, so that again $\fb{u_0}$ does not intersect $B_r$.

\smallskip 

Reasoning as above and using a covering argument one can show that, for a fixed compact set $D$, 
the quantity $\delta$ in the definition of 
$d_{\mathcal H}$, with $G=\partial\{u_0>0\}\cap D$ and $F=\partial\{u_k>0\}\cap D$, 
can be chosen as small as we wish. 

The last statement follows from the non-degeneracy of $u^+$ given by Corollary \ref{cor-nondeg}, 
the convergence of $\fb{u_k}\to \fb{u_0}$ in Hausdorff distance
and the fact that the $N$-dimensional Hausdorff measure $\mathcal H^{N}(\fb{u_0})=0$, 
since $u_0$ is also minimizer and Corollary~\ref{coro:zero} applies. 
Hence the proof of Proposition \ref{tech-2} is concluded.  
\end{proof}

\begin{remark}\label{multi-blow}
In view of Proposition \ref{str-nd-hi} we see that when we consider the blow-up 
of a minimizer, the limit cannot vanish, no matter how many times we blow-up the minimizer 
$u$ at a non-flat point.
\end{remark}

We now finish the proof of Theorem B. More precisely, we show that 
\begin{equation}\label{usouno-bis}
\mathcal H^{N-1}\left(( \fb{ u}\setminus\fbr u) \cap B_R\right)=0.
\end{equation}

First observe that $\mathscr N\subset \partial\{u>0\}\setminus \partial_{\rm red}\{u>0\}$, 
see the discussion in Section \ref{sec:eps}.
Since the current boundary $T:=\partial(\R^N\L\{u>0\}\cap B_R(0))$ is 
representable by integration, $\|T\|=\int_{B_R(0)} |D\X{u>0}|$, we get
from Section 4.5.6. on page 478 of \cite{Federer} that 
\begin{equation}\begin{split}\label{rew21}
&{\mbox{$\partial\{u>0\}\setminus \partial_{\rm red}\{u>0\}=K_0\cup K_+ $, 
where $\mathcal H^{N-1}(K_+)=0$}}\\& {\mbox{and for 
$x_1\in K_0$,  $r^{1-N}\mathcal H^{N-1}(\partial_{\rm red}\{u>0\}\cap B_r(x_1))\rightarrow 0$ 
as $r\rightarrow 0$.}}
\end{split}\end{equation}

Let us show that 
\begin{equation}\label{rew43}
K_0=\emptyset.\end{equation}
To see this, for $k\in\N$, we define $u_{k}(x):=\frac{u(x_1+r_k x)}{r_k}$, 
where $r_k\to0$ as $k\to+\infty$. 
By the compactness properties obtained in Proposition \ref{tech-2}, 
we have that $u_{k}\rightarrow u_{0}$, as $k\to+\infty$, for some function $u_0$ 
and, for any test function $\varphi$,
\begin{eqnarray*}
\int_{B_R}\X{u_0>0}\div\phi\longleftarrow\int_{B_R}\X{u_{k}>0}\div \phi=r^{1-N}_k\int_{B_{Rr_k}(x_1)}\X{u>0}\div\phi\left(\frac{x-x_1}{r_k}\right)\\\nonumber
\leq \sup \phi \,r_k^{1-N}\mathcal H^{N-1}(\partial_{\rm red }\{u>0\}\cap B_{Rr_k}(x_1))
\longrightarrow 0\qquad {\rm{as}}\qquad k\rightarrow +\infty,
\end{eqnarray*}
where \eqref{rew21} was also used. 

Hence we infer that $\X{u_0>0}$ is a function of bounded variation which is constant a.e. in $B_R$.
The positive Lebesgue density property of $\{u\le 0\}$ obtained in Lemma \ref{density est} 
and translated to $u_0$ 
through compactness, and the strong maximum principle for $p-$harmonic functions
demand $u_0$ to be zero. 
This is in contradiction with the non-degeneracy of $u^+$ stated by Proposition \ref{str-nd-hi} 
(notice that, by a compactness argument, the non-degeneracy property translates to $u_0$). 
Thus \eqref{rew43} is proved. 

From \eqref{rew21} and \eqref{rew43} we obtain that 
$\mathcal H^{N-1}\left((\partial\{u>0\}\setminus \partial_{\rm red}\{u>0\})\cap B_R\right)=0$. 
The proof of Theorem B is then finished. 

%
%

\section{Proof of Theorem C}\label{sec-ssuka} 

With the aid of Theorem A, in this section we complete the proof of Theorem C. 

\begin{proof}[Proof of Theorem C] 
It is well-know that in order to prove the estimate \eqref{ssuka} 
it is enough to show that $u$ grows linearly away from the free boundary.
For this, let $0\in \fb u$ and $B_{\frac14}\Subset \Omega$. 
Notice that, if for all $k\in \mathbb N$, $k\ge 2$, we have that 
$h(0, 2^{-i})\ge h_02^{-i-1}$, then it follows from Theorem A that 
$\sup_{B_r} |u|\leq 2Lr$. Therefore, suppose that there is $k_0\in \mathbb N$ such that 
\begin{equation}\label{ssuka-1}
h\left(0, \frac1{2^j}\right)\ge \frac{h_0}2\frac1{2^j}, \quad j=2, \dots, k_0-1, 
\end{equation}
but
\begin{equation}\label{ssuka-2}
h\left(0, \frac1{2^{k_0}}\right) < \frac{h_0}2\frac1{2^{k_0}}.
\end{equation}
From \eqref{ssuka-1} and Proposition \ref{prop:lin flat} (or Corollary \ref{cor:lin flat}) it follows that 
\begin{equation}\label{ssuka-3}
\sup_{B_{\frac1{2^{k_0-1}}}}|u|=\sup_{B_{\frac12\frac1{2^{k_0-2}}}}|u|\leq 2L\frac1{2^{k_0-2}}=\frac{4L}{2^{k_0-1}}.
\end{equation}
Denote $R_0:=\frac1{2^{k_0-1}}$ and introduce
\begin{equation}\label{ssuka-4}
v_0(x):=\frac{u(R_0x)}{R_0}, \quad x\in B_1, 
\end{equation}
then by \eqref{mek-qash} it follows that $v_0$ is a minimizer in $B_1$. 
Furthermore, \eqref{ssuka-3} yields
\begin{equation}\label{ssuka-5}
\sup_{B_1}|v_0|\leq 4L
\end{equation}
and by \eqref{ssuka-2} we see that $\fb {v_0}\cap B_{\frac12}$ is $h_0/2$ flat. 
Therefore, we infer from the second part of  Theorem A that there are 
$\delta\in(0, \frac12)$ and $\alpha>0$ depending on $N$, $p$, $\Lambda$, $h_0$ and $4L$ such that 
$\fb {v_0}\cap B_{\delta}$ is $C^{1, \alpha}$ regular. Applying the boundary gradient estimates  for $p-$harmonic functions
we finally obtain 
\begin{equation}\label{ssuka-6}
\sup_{B_{\frac{\delta}{2}}}|\na v_0^\pm(x)|\leq 4LC_0
\end{equation}
for some tame constant $C_0>0$. Recalling \eqref{ssuka-4}
and \eqref{ssuka-3} we conclude that 
$$\sup_{B_{r}}|u|\leq \frac{16 L}{\delta} r, \quad \forall r<r_0$$
for some small  universal constant $r_0$. This completes the proof of Theorem C.
\end{proof}

\appendix

\section{Viscosity solutions and linear development}\label{ap:lemma}

Here we show Lemma \ref{lemma:linear}.

\begin{proof}[Proof of Lemma \ref{lemma:linear}]
We first show~a). Without loss of generality, we may assume that~$x_0=0$
and~$\nu=e_N$. Let~$B:=B_R(y_0)$ be a touching ball at~$0\in \Gamma$,  for some~$y_0\in\Omega$ and~$R>0$.

Now, we want to establish~\eqref{linear alpha}.
For this, we first construct a function that can be used as a barrier to control~$u$ from below in the ring~$B_R(y_0)\setminus B_{R/2}(y_0)$. 
We consider the scaled $p-$capacitary function $H$, that is 
$p$-harmonic in~$B_R(y_0)\setminus B_{R/2}(y_0)$, that 
vanishes on~$\partial B_R(y_0)$ and that is equal to~1 on~$\partial B_{R/2}(y_0)$.
Observe that near the origin
\begin{equation}\label{origin}
H(x)=c(N,R)\,x_N +o(|x|),
\end{equation}
for some~$c(N,R)>0$.

Using the Harnack inequality we see that $u(x)\ge c_0u(y_0)$ in 
$\bar{B}_{R/2}(y_0)$, for some~$c_0>0$. Thus, 
multiplying $H$ with a suitable constant $\sigma:=c_0u(y_0)$ we obtain 
that $\sigma H\le u$ on~$\partial B_{R/2}(y_0)$. 
Moreover, $u\ge 0=\sigma H$ on~$\partial B_R(y_0)$. 
Hence, by comparison principle, we get that
$$ u(x)\ge \sigma H(x) \quad {\mbox{ in }}B_R(y_0)\setminus B_{R/2}(y_0). $$
From this and~\eqref{origin} we obtain that
\begin{equation}\label{pqoeuiqwtu}
u(x)\ge \sigma c(N,R)x_N +o(|x|)
\end{equation}
near the origin. 

Now we take~$k_0$ the smallest positive integer such that~$2^{-k_0}\le R/2$ 
and we define 
\begin{equation}\label{Ak}
A_k:=\sup\{m: u(x)\ge mx_N {\mbox{ in }} B_{2^{-(k_0+k)}}\cap B_R(y_0)\},\qquad k=0, 1, 2, \dots.
\end{equation}
Thanks to~\eqref{pqoeuiqwtu} the set of numbers~$m$ in the definition 
of~$A_k$ is not empty. 
Notice also that the sequence~$\{A_k\}$ is increasing,
and so we let~$A:=\sup A_k$.

We observe that
\begin{equation}\label{alpha pos}
A >0. \end{equation}
Indeed, since~$u(x)\ge\sigma  H(x)$
in~$\overline{B_R(y_0)\setminus B_{R/2}(y_0)}$ then~$A_0>0$.
This implies~\eqref{alpha pos}, because ~$A_k$ is 
increasing.

If~$A=\infty$, then~$u$ grows faster than any linear function at~$0$.
While, if~$A<\infty$, then~\eqref{linear alpha} holds true.

Now we claim that equality in~\eqref{linear alpha} holds in any non-tangential
domain. In what follows we denote by
\begin{equation}\label{sets}
\mathcal B:=B_s(e_N/2), \ {\mbox{ for some small $s>0$, }}
\ \mathcal D_k:=B_{R/{r_k}}(y_0/r_k)\cap B_1 \  {\mbox{ and }} \ r_k:=2^{-(k+k_0)}.\end{equation}
If the claim fails  then there exist
a sequence of points~$x^k\in B_R(y_0)$ and~$\delta_0>0$ such that
\begin{equation}\label{points}
u(x^k)>A x_N^k +\delta_0 |x^k| \quad {\mbox{ and }} \quad
|x_k|=r_k\sim \dist(x^k,\partial B_R(y_0)).
\end{equation}
Now let~$u_k(x):=\frac{u(r_kx)}{r_k}$.
Notice that~\eqref{points} implies that
\begin{equation*}
u_k(y^k)>A y_N^k +\delta_0,
\end{equation*}
where $y^k:=x^k/r_k\in\partial B_1\cap B_{R/r_k}(y_0/r_k)$.
This implies that 
\begin{equation}\label{gfuwefhiw}
u_k(x)-A x_N\ge c_0\,\delta_0
\end{equation}
on some fixed portion of~$\partial B_{1}\cap B_{R/r_k}(y_0/r_k)$,
for some~$c_0>0$. So, \eqref{gfuwefhiw} and the Harnack inequality give that
\begin{equation}\label{ha3}
u_k(x)-A x_N\ge \frac{c_0\delta_0}{100} \quad {\mbox{ in }} \mathcal B,
\end{equation}
where~$\mathcal{B}$ has been introduced in~\eqref{sets} and we can take 
$s=\frac18$.

Since~$u_k$ are uniformly~$C^{1,\alpha}_{loc}$ in $B_2\cap B_{R/r_k}(y_0/r_k)$
and uniformly continuous in $B_2\cap \{x_N>0\}$, we have that,
up to a subsequence, $u_k$ converges uniformly to some~$u_0\ge0$
in~$B_1\cap\{x_N\ge0\}$. Therefore, by construction of $A$,
\begin{equation}\label{ha2}
u_0(x)-A x_N\ge 0 \quad {\mbox{ in }} \partial B_1\cap \{x_N\ge0\}.
\end{equation}

We recall~\eqref{sets} and define functions~$w_k$ 
as solutions to the following boundary value problem
\begin{equation}\label{eq w-k}
\left\{\begin{array}{lll}
\Delta_p w_k=0 & \ {\mbox{ in }}\mathcal D_k\setminus \mathcal B,\\
w_k=A_kx_N+\frac{c_0\delta_0}{200} & \ {\mbox{ on }}\partial \mathcal B, \\ 
w_k=A_k x_N & \ {\mbox{ on }}\partial\mathcal D_k.
\end{array}
\right.
\end{equation}
Now from the definition of~$A_k$ in~\eqref{Ak} we have 
that~$u_k(x)\ge A_kx_N$ in~$B_{R/r_k}(y_0/r_k)\cap B_{1}=\mathcal D_k$, and so~$w_k=A_kx_N\le u_k$ on~$\partial \mathcal D_k$. 
Moreover, on $\partial \mathcal B$ we have that $w_k=A_kx_N+\frac{c_0\delta_0}{200}\le Ax_N+\frac{c_0\delta_0}{200}\le u_k$, thanks to~\eqref{ha3}. By comparison principle 
we get that $w_k\leq u_k$ in $\mathcal D_k\setminus \mathcal B$.

By construction  $w_k\to w_0$ uniformly in $B_\mu\cap \mathcal D_k$ and 
\begin{equation}
\|w_k-w_0\|_{L^\infty(\mathcal D_k)}\leq \epsilon_k \to 0.
\end{equation} 
From the stability of $C^{1,\alpha}$ norm in $B_{3/16}^+$ (recall that we chose $s=1/8$ in \eqref{sets}) we conclude 
\begin{equation}
\|\nabla(w_k-w_0)\|_{C^\alpha(B_{3/16}^+\cap \mathcal D_k)}\leq \epsilon_k.
\end{equation}
This allows to estimate the H\"older norm of $\nabla (w_k-w_0)$ near the flat portion of the 
boundary of $B^+_{3/16}$.

By Hopf's Lemma there is $\gamma>0$ such that $w_0\ge (A+\gamma)x_N+o(|x|)$ near 
the origin. Combining we get that 
\begin{eqnarray}
u_k&\ge& w_k=w_k-w_0+w_0\ge w_k-w_0+(A+\gamma)x_N+o(|x|)\\\nonumber
&\ge &(A+\gamma)x_2+o(|x|) -\epsilon_k x_N \\\nonumber
&\ge& (A+\frac{\gamma}2)x_N+o(|x|) \\\nonumber
&\ge&   (A+\frac{\gamma}4)x_N\\\nonumber
\end{eqnarray}
in $B_{1/2}^+\cap B_{R/r_k}(y_0/r_0)$. Returning to $u$ we get that 
$$u(x)\ge (A+\frac{\gamma}4)x_N\ge (A_{k+1}+\frac{\gamma}8)x_N $$ 
in $B_{r_k/2}\cap B_R(y_0)=B_{r_{k+1}}\cap B_R(y_0)$. This is a contradiction with the definition 
of~$A_k$ in~\eqref{Ak}. 

Hence, \eqref{linear alpha} holds true in any non-tangential domain,
and this concludes the proof of part~a).

\smallskip

Now we show part~b). For this, we take a ball~$B_R(y_0)$ touching~$x_0$
from outside~$\Omega$. We construct the barrier as follows:
we let~$\eta$ to be a~$p$-harmonic function in~$B_{2R}(y_0)\setminus B_R(y_0)$,
such that~$\eta=0$ on~$\partial B_R(y_0)$ 
and~$\eta=\max_{\partial B_{2R}(y_0)}u$
on~$\partial B_{2R}(y_0)$. Then, from comparison principle we have that~$u\le\eta$
in~$B_{2R}(y_0)\cap\Omega$.
Moreover, by Hopf's Lemma
\begin{equation}\label{dfwihfweifhevsas}
\eta(x)=C(N,R)x_N +o(|x|)
\end{equation}
near the origin, for some~$C(N,R)>0$.

We take~$k_0$ to be the smallest positive integer such that~$2^{-k_0}<R/2$,
and we define
$$ \beta_0:=\inf\{m: m\eta(x)\ge u(x) {\mbox{ in }}B_{2^{-k_0}}\cap B_R^c(y_0)\},$$
and, for any~$k\ge 1$,
$$ \beta_k:=\inf\{m: m\eta(x)\ge u(x) {\mbox{ in }}B_{2^{-(k_0+k)}}\cap B_R^c(y_0)\}.$$
Since~$\beta_k$ is a decreasing sequence, we can take~$\tilde{\beta}:=\inf\beta_k$.
Hence, $\tilde{\beta}\ge0$, and, setting~$\beta:=\tilde{\beta}C(N,R)$,
from~\eqref{dfwihfweifhevsas} we deduce~\eqref{linear beta}.

In order to prove equality in~\eqref{linear beta} in every non-tangential domain,
one can proceed as in the proof of part~a).
This concludes the proof of Lemma~\ref{lemma:linear}.
\end{proof}

%
%


\begin{thebibliography}{widest-label}

\bibitem{ACF} {\sc H.W. Alt, L.A. Caffarelli, A. Friedman}:
Variational problems with two phases and their free boundaries.
{\it Trans. Amer. Math. Soc.} {\bf 282} (1984), no. 2, 431--461.

\bibitem{ACF-quasi} {\sc H.W. Alt, L.A. Caffarelli, A. Friedman};
A free boundary problem for quasilinear elliptic equations. 
{\it Ann. Scuola Norm. Sup. Pisa Cl. Sci. (4)} {\bf 11} (1984), no. 1, 1--44.

\bibitem{Ast}{\sc G. Astarita, G. Marrucci}: {\em 
Principles of non-Newtonian fluid mechanics}. MacGraw Hill, London, New York, 1974.

\bibitem{BZ} {\sc G. Birkhoff, E.H. Zarantonello}: {\em Jets, Wakes, and Cavities}. 
Academic Press, 1957.

\bibitem{Moreira} {\sc J.E. Braga, D.R. Moreira}: Uniform Lipschitz regularity for classes of minimizers in two phase free 
boundary problems in Orlicz spaces with small density on the negative phase. 
{\it Ann. Inst. H. Poincar\'e Anal. Non Lin\'eaire} {\bf 31} (2014), no. 4, 823--850. 

\bibitem{Caffa1} {\sc L.A. Caffarelli}: 
A Harnack inequality approach to the regularity of free boundaries. I. 
Lipschitz free boundaries are $C^{1,\alpha}$. 
{\it Rev. Mat. Iberoamericana} {\bf 3} (1987), no. 2, 139--162.

\bibitem{Caffa3} {\sc L.A. Caffarelli}: A Harnack inequality approach to the regularity 
of free boundaries. III. Existence theory, compactness, and dependence on $X$. 
{\it Ann. Scuola Norm. Sup. Pisa Cl. Sci. (4)} {\bf 15} (1988), no. 4, 583--602.

\bibitem{Caffa2} {\sc L.A. Caffarelli}: A Harnack inequality approach to the regularity 
of free boundaries. II. Flat free boundaries are Lipschitz. 
{\it Comm. Pure Appl. Math.} {\bf 42} (1989), no. 1, 55--78. 

\bibitem{Luis} {\sc L.A. Caffarelli, S. Salsa}:
{\it A geometric approach to free boundary problems}.
Graduate Studies in Mathematics, 68. American Mathematical Society,
Providence, RI, 2005. x+270 pp.

\bibitem{Cianchi} {\sc A. Cianchi}: 
Continuity properties of functions from Orlicz-Sobolev spaces and embedding theorems. 
{\it Ann. Scuola Norm. Sup. Pisa Cl. Sci. (4)} {\bf 23} (1996), no. 3, 575--608.

\bibitem{DP2} {\sc D. Danielli, A. Petrosyan}:
A minimum problem with free boundary for a degenerate quasilinear operator.
{\it Calc. Var. Partial Differential Equations} {\bf 23} (2005), no. 1, 97--124.

\bibitem{DPS} {\sc D. Danielli, A. Petrosyan, H. Shahgholian}:
A singular perturbation problem for the $p$-Laplace operator.
{\it Indiana Univ. Math. J.} {\bf 52} (2003), no. 2, 457--476.

\bibitem{DiB-M} {\sc E. DiBenedetto, J. Manfredi, }
On the higher integrability of the gradient of weak solutions 
of certain degenerate elliptic systems.
{\it Amer. J. Math.} {\bf 115} (1993), no. 5, 1107--1134.

\bibitem{Diening} {\sc L. Diening, B. Stroffolini, A. Verde}: 
Everywhere regularity of functionals with $\phi$-growth. 
{\it Manuscripta Math.} {\bf 129} (2009), no. 4, 449--481. 

\bibitem{DK2} {\sc S. Dipierro, A.L. Karakhanyan}: A new discrete 
monotonicity formula with application to a two-phase free boundary problem 
in dimension~2. {\it Preprint}, http://arxiv.org/abs/1509.00277

\bibitem{Duzaar} {\sc F. Duzaar, G. Mingione}: The $p$-harmonic approximation and the regularity 
of $p$-harmonic maps. 
{\it Calc. Var. Partial Differential Equations} {\bf 20} (2004), no. 3, 235--256.

\bibitem{Federer} {\sc H. Federer;}
{\it Geometric measue theory}. Springer-Verlag, Berlin, Heidelberg and New York, 1969.

\bibitem{Giaq} {\sc M. Giaquinta}: 
{\it Multiple integrals in calculus of variations and nonlinear elliptic systems}. 
Annals of Mathematics Studies, 105. Princeton University Press, Princeton, NJ, 1983. vii+297 pp.

\bibitem{GM} {\sc S. Granlund, N. Marola}:
On the problem of unique continuation for the $p$-Laplace equation.
{\it Nonlinear Anal.} {\bf 101} (2014), 89--97.

\bibitem{Gur} {\sc M.I. Gurevich}: {\em The theory of jets in an ideal fluid}.
Translated from the Russian by R. E. Hunt. Translation edited by E. E. Jones and G. Power. 
International Series of Monographs in Pure and Applied Mathematics, Vol. 93 Pergamon Press, Oxford-New York-Toronto, Ont. 1966 viii+412 pp.

\bibitem{Aram JDE} {\sc A.L. Karakhanyan}: Up-to boundary regularity for a singular perturbation problem of $p$-Laplacian type. 
{\it J. Differential Equations} {\bf 226} (2006), no. 2, 558--571.

\bibitem{K1} {\sc A.L. Karakhanyan}: On the Lipschitz regularity of solutions of 
minimum problem with free boundary. {\it Interfaces Free Bound.} {\bf 10} (2008), no. 1, 79--86.

\bibitem{LN1} {\sc J.L. Lewis, K. Nystr\"om}: Regularity of Lipschitz free boundaries in 
two-phase problems for the~$p$-Laplace operator. {\it Adv. Math.} {\bf 225} (2010),
2565--2597.

\bibitem{LN2} {\sc J.L. Lewis, K. Nystr\"om}: Regularity of flat free boundaries in two-phase problems for the~$p$-Laplace operator. 
{\it Ann. Inst. H. Poincar\'e Anal. Non Lin\'eaire } {\bf 29} (2012),
83--108.

\bibitem{MZ} {\sc J. Mal\'y, W.P. Ziemer}: {\it Fine regularity of solutions of elliptic partial differential equations}. 
Mathematical Surveys and Monographs, 51. American Mathematical Society, Providence, RI, 1997. xiv+291 pp.

\bibitem{Philip} {\sc J.R. Philip}: $n$-diffusion.
{\em Austral. J. Phys.} {\bf 14} (1961) 1--13. 

\bibitem{T1} {\sc P. Tolksdorf}: On the Dirichlet problem for quasilinear equations in domains 
with conical boundary points. {\it Comm. Partial Differential Equations} {\bf 8} 
(1983), no. 7, 773--817.


\end{thebibliography}
\end{document}